\documentclass[reqno]{amsart}
\usepackage{xcolor}
\usepackage{amsmath, amssymb}
\usepackage{latexsym}
\usepackage{wasysym}
\usepackage{graphicx}

\newtheorem{Theorem}{Theorem}[section]
\newtheorem{lemma}[Theorem]{Lemma}
\newtheorem{Proposition}[Theorem]{Proposition}
\newtheorem{Corollary}[Theorem]{Corollary}
\newtheorem{remark}[Theorem]{Remark}

\numberwithin{equation}{section}

\newcommand{\la}{\langle}
\newcommand{\ra}{\rangle}

\newcommand{\beq}{\begin{equation}}
\newcommand{\eeq}{\end{equation}}
\newcommand{\bes}{\begin{equation*}}
\newcommand{\ees}{\end{equation*}}

\def\d{\displaystyle}

\def\N{{\Bbb N}}
\def\Z{{\Bbb Z}}
\def\R{{\Bbb R}}
\def\T{{\Bbb T}}
\def\C{{\Bbb C}}

\hfuzz=\maxdimen
\begin{document}

\title{1d Quantum Harmonic Oscillator Perturbed by a Potential with Logarithmic Decay}
\author{ Zhiguo Wang and Zhenguo Liang}

\address {School of Mathematical Sciences,
Soochow University, Suzhou
215006,  China} \email{zgwang@suda.edu.cn}

\address {School of Mathematical Sciences and
Key Lab of Mathematics for Nonlinear Science, Fudan University,
Shanghai 200433, China} \email{zgliang@fudan.edu.cn}

\thanks{The first author was partially supported by NSFC grants  11571249, 11271277; the second author was partially supported by NSFC grants 11571249, 11371097.}


\date{}

\begin{abstract} {In this paper we prove an infinite dimensional KAM theorem, in which the assumptions on the derivatives of perturbation in \cite{GT} are weakened from \textsl{polynomial decay} to \textsl{logarithmic decay}.  As a consequence, we apply it  to 1d quantum harmonic oscillators and prove  the reducibility of a linear harmonic oscillator, $T=-\frac{d^2}{dx^2}+x^2$, on $L^2(\R)$ perturbed by a quasi-periodic in time potential $V(x,\omega t; \omega)$ with \textsl{logarithmic decay}. This entails the
pure-point nature of the spectrum of the Floquet operator $K$, where
\begin{eqnarray*}
K:=-{\rm i}\sum_{k=1}^n\omega_k\frac{\partial}{\partial \theta_k}-\frac{d^2}{dx^2}+x^2+\varepsilon V(x,\theta;\omega),
 \end{eqnarray*}
 is defined on $L^2(\R) \otimes L^2(\T^n)$ and the potential $V(x,\theta;\omega)$ has  logarithmic decay as well as its gradient in $\omega$.}
\end{abstract}

\maketitle

\section{Introduction and Main Results}\label{introduction}
\subsection{Statement of the Results.}
In this paper we consider the linear equation
\begin{eqnarray}\label{maineq}
{\rm i}\partial_tu=-\partial_x^2 u+x^2u+\varepsilon V(x,\omega t;\omega)u,\ \ \ u=u(t,x),\ x\in\R,
 \end{eqnarray}
where $\varepsilon>0$ is a small parameter and the frequency vector $\omega$ of forced oscillations is regarded as a parameter in $\Pi:= [0,2\pi)^n$.
 We assume that the potential $V: \R\times \T^n \times \Pi\ni( x,\theta;\omega)\mapsto
\R$ is $C^3$ smooth in all its variables and analytic in $\theta$ where $\T^n=\R^n/2\pi\Z^{n}$ denotes the $n$-dimensional torus.  For $\rho>0$, the function $V( x,\theta;\omega)$ analytically in $\theta$ extends to the domain
$$\T_{\rho}^n = \{(a+bi)\in \C^{n}/2\pi\Z^n:\ |b|<\rho\}$$ 
as well as its gradient in $\omega$
and satisfies
\begin{eqnarray}
|V(x,\theta;\omega)|,\ |\partial_{\omega_j}V( x,\theta;\omega)| &\leq& C(1 + \ln(1+x^2))^{-2\beta},\label{VV1}\\
\ |\partial_xV(x,\theta;\omega)|,\ |\partial_x \partial_{\omega_j}V(x,\theta;\omega)|&\leq& C,\label{VV2}\\
\ |\partial^2_{x}V(x,\theta; \omega)|,\ |\partial^2_{x}\partial_{w_j}V(x,\theta;\omega) |&\leq& C,\label{VV3}
 \end{eqnarray}
 where $(x, \theta; \omega)\in \R\times \T_{\rho}^n\times \Pi$,  $\beta\geq 2(n+2)$, $j=1,\cdots, n$ and $C>0$.
\begin{Theorem}\label{maintheorem}
Assume that $V$ satisfies (\ref{VV1}) - (\ref{VV3}) and $\beta\geq 2(n+2)$. Then there exists $\varepsilon_0$ such that for all $0 \leq\varepsilon<\varepsilon_0$ there exists $\Pi_\varepsilon\subset [0, 2\pi)^n$ of positive measure and asymptotically full measure:
$Meas(\Pi_\varepsilon)\rightarrow(2\pi)^n$ as $\varepsilon\rightarrow 0$, such that for all $\omega\in\Pi_\varepsilon$, the linear Schr\"odinger
equation (\ref{maineq})
reduces, in $L^2(\R)$, to a linear equation with constant coefficients (with respect to the
time variable).
 \end{Theorem}
 Similar as \cite{GT}, the above theorem has two direct corollaries. As a preparation we define the harmonic oscillator $T=-\frac{d^2}{dx^2}+x^2$ and its related Sobolev space. Let $p\geq 2$ and denote  $\ell^{2,p}$ be the Hilbert space  of all real  $w=(w_j)_{j\geq 1}$
with
$$\|w\|_{p}^2 = \sum\limits_{j\geq 1} j^p|w_j|^2<\infty.$$
The operator $T$ has eigenfunctions $(h_j)_{j\geq 1}$, so called  the Hermite functions, which  satisfy
\begin{eqnarray}\label{eigensection1}
Th_j=(2j-1)h_j, \qquad \|h_j\|_{L^2(\R)}=1,\qquad  j\geq 1,
\end{eqnarray} and form a Hilbertian basis of $L^2(\R)$. Let $u=\sum\limits_{j\geq 1}u_jh_j$ be a typical element of $L^2(\R)$. Then $(u_j)\in \ell^{2,p}$ if and only if $$u\in \mathcal{H}^p : = D(T^{p/2})=\{u\in L^2(\R):\  x^{\alpha_1}\partial^{\alpha_2}u \in L^{2}(\R)\ {\rm for}\ 0\leq\alpha_1+\alpha_2\leq p \}.$$
\indent For a function $f\in \mathcal{H}^p(\R)$ we define $$\|f\|_{\mathcal{H}^p}^2=\sum\limits_{0\leq \alpha_1+\alpha_2\leq p}\|x^{\alpha_1}\partial_x^{\alpha_2}f\|_{L^2}^2<\infty.$$
Then we have
 \begin{Corollary}\label{coro1}
Assume that $V$ and $\partial_{\omega_j}V,\ j=1,\cdots, n$ are  $C^\infty$ in $x$ and there exists a constant $C>0$ such that for all $\nu \geq 1$, $x\in \R$ and $|\Im \theta|<\rho$,
\begin{eqnarray*}\label{forarbitraryHpnorm1}
|V(x,\theta;\omega)|,\ |\partial_{\omega_j}V(x,\theta;\omega)| &\leq& C(1 + \ln(1+x^2))^{-2\beta},\\
|\partial_x^{\nu}V(x,\theta; \omega)|,\ |\partial_x^{\nu}V_{\omega_j}(x,\theta; \omega)|&\leq& C.\label{forarbitraryHpnorm2}
\end{eqnarray*}
Let $p \geq0$ and $u_0\in \mathcal{H}^p$ and $\beta\geq 2(n+2)$. Then there exists $\varepsilon_0 > 0$ so that for all $0 \leq\varepsilon<\varepsilon_0$
and $\omega\in\Pi_\varepsilon$, there exists a unique solution $u \in \mathcal{C}(\R,\mathcal{H}^p)$ of (\ref{maineq}) so that $u(0) = u_0$.
Moreover, $u$ is almost-periodic in time and we have the bounds
$$(1-\varepsilon C)\|u_0\|_{\mathcal{H}^p} \leq \|u(t)\|_{\mathcal{H}^p} \leq (1 + \varepsilon C)\|u_0\|_{\mathcal{H}^p} , \ {\rm with}\ t \in \R,$$
for some $C = C(p, \omega)$.
 \end{Corollary}

 Consider  the Floquet operator on $  L^2(\R)\otimes L^2(\T^n)$
\begin{eqnarray}\label{KK}
K:=-{\rm i}\sum_{k=1}^n\omega_k\frac{\partial}{\partial \theta_k}-\frac{d^2}{dx^2}+x^2+\varepsilon V(x,\theta;\omega),
 \end{eqnarray}
 then we have
 \begin{Corollary}\label{coro2}
Assume that $V$  satisfies the same conditions as in Theorem \ref{maintheorem} and $\beta\geq 2(n+2)$. There exists $\varepsilon_0 > 0$ so that for all $0 \leq\varepsilon<\varepsilon_0$
and $\omega\in\Pi_\varepsilon$, the spectrum of the Floquet operator $K$ is pure point.
 \end{Corollary}
\subsection{Related results.}
As in \cite{BG} the equations (\ref{maineq})  can be generalized into a time-dependent Schr\"odinger equation
\begin{eqnarray}\label{NLStime3}
 {\rm i} \partial_{t}\psi(t)=(A+\epsilon P(\omega t))\psi(t),
 \end{eqnarray}
where $A$ is a positive self-adjoint operator on a separable Hilbert space $\mathcal{H}$ and the perturbation $P$ is an operator-valued function  from $\T^n$ into the space of
symmetric operators on $\mathcal{H}$. The Floquet operator associated with (\ref{NLStime3}) is defined by
\begin{eqnarray*}
K_{F} : = -{\rm i}\omega \cdot  \partial_{\theta}+ A +\epsilon P(\theta)\quad {\rm on}\ \mathcal{H} \otimes L^2(\T^n).
\end{eqnarray*}
\indent It is well-known that long-time behavior of the solution $\psi(t)$ of the time-dependent Schr\"odinger equation (\ref{NLStime3})
is closely related to the spectral properties of the Floquet operator $K_{F}$(see Wang \cite{Wang}). It has been proved in \cite{BEL, BLE, DS, DSV, GY, H, JOY, N} that the Floquet operator $K_{F}$ is of pure point spectra or no absolutely continuous spectra where $P$ is bounded. When $P$ is unbounded,  the first result was obtained by Bambusi and  Graffi \cite{BG} where they considered  the time dependent Schr\"odinger equation
\begin{eqnarray}\label{NLStime1}
{\rm i}\partial_{t} \psi(x,t) = H(t) \psi(x,t), x\in \R;  \qquad H(t) : = -\frac{d^2}{dx^2} + Q(x) +\epsilon V(x,\omega t),\ \epsilon\in \R,
\end{eqnarray}
where $Q(x) \sim |x|^{\alpha}$ with $\alpha>2$ as $|x|\rightarrow \infty$ and $|V(x,\theta)||x|^{-\beta}$ is bounded as $|x|\rightarrow \infty$ for some $\beta<\frac{\alpha-2}{2}$.
This entails the pure-point nature of the  spectrum of  the Floquet operator
\begin{eqnarray*}\label{floquetspectrum1}
K_{F} : = -{\rm i}\omega \cdot \partial_{\theta}-\frac{d^2}{dx^2}+Q(x)+\varepsilon V(x,\theta),
\end{eqnarray*}
on $L^2(\R) \otimes L^2(\T^n)$ for $\varepsilon$ small. Liu and Yuan \cite{LiuYuan0} solved the limit case when $\alpha>2$ and $\beta\leq \frac{\alpha-2}{2}$, which can be applied to the so-called quantum version of the Duffing oscillator
\begin{eqnarray*}\label{NLStime2}
{\rm i}\partial_{t} \psi(x,t) = \big(-\frac{d^2}{dx^2} +x^4 +\epsilon xV(\omega t)\big) \psi(x,t), \qquad x,\epsilon\in \R.
\end{eqnarray*}
\indent The results in \cite{BG} and \cite{LiuYuan0} didn't include the case $Q(x)=x^2$(see (\ref{NLStime1})), which is so-called quantum harmonic oscillator($\alpha=2$). The quantum harmonic oscillator is the quantum-mechanical analog of the classical harmonic oscillator. Because an arbitrary potential can usually be approximated as a harmonic potential at the vicinity of a stable equilibrium point, it is one of the most important model systems in quantum mechanics. \\
\indent In \cite{EV}  Enss and Veselic proved that, if $\omega$ is rational,  the Floquet operator   $K$ (see (\ref{KK})) has pure point spectrum when the perturbing potential $V$ is bounded and has sufficiently fast decay at infinity.  In \cite{Wang} Wang proved the spectrum of the Floquet operator $K$ is pure point where the perturbing potential 
$V(x,\theta) = e^{-x^2}\sum\limits_{k=1}^{n}\cos \theta_{k}$ has
\textit{exponential decay}. Greb\'ert and Thomann \cite{GT} improved the results in \cite{Wang} from exponential decay to \textit{polynomial decay}.
 {In this paper we improve the results in \cite{GT} from polynomial decay to \textit{logarithmic decay}. But we know nothing about $K_{F}$ when $\alpha=2$ and $\beta=0$, i.e.  $Q(x) \sim |x|^2$ as $|x|\rightarrow \infty$ and $V(x,\theta)$ is only bounded(except the special case when $V(x,\theta)$ is independent of $x$, see \cite{GT}). This problem was firstly posed by Eliasson in \cite{E1} and is still open till now. }\\
\indent For $V$ is unbounded it is a different situation. See \cite{ GY, HLS, YK} and related discussions in \cite{Wang}.
\subsection{The main idea for proving Theorem \ref{maintheorem}.}  As in \cite{BG}, \cite{EK0} and \cite{GT} the proof of Theorem \ref{maintheorem} is closely related to an infinite dimensional KAM theorem.  Since the formulation
of this abstract theorem is technical and lengthy, we postpone it to Sect. \ref{S2}, see Theorem \ref{KAMtheorem}(KAM Theorem). Let us firstly explain the main idea and techniques for proving Theorem \ref{KAMtheorem}.\\
\indent We begin with a parameter dependent family of analytic Hamiltonians of the form
 \begin{eqnarray*}\label{}
H &=& N(y, z,\bar{z}; \xi)+P(\theta,y,z,\bar{z};\xi)\\
&= & \sum_{1\leq j\leq n}
\omega_j(\xi)  y_j + \sum_{j\geq 1}
\Omega_j(\xi)z_j\bar{z}_j+P(\theta,y,z,\bar{z};\xi),
 \end{eqnarray*}
where $(\theta,y)\in\T^n\times\R^n,$ $z=(z_j)_{j\geq1},\bar{z}=(\bar{z}_j)_{j\geq1}$ are infinitely many variables, $\omega(\xi)=(\omega_j(\xi))\in\R^n$, $\Omega(\xi)=(\Omega_j(\xi))\in\R^\infty$ and $\xi\in \Pi\subset\R^n.$
For our applications we suppose $\Omega_j(\xi)=2j-1$ for simplicity and $\la l,\Omega(\xi)\ra\neq0,\ \forall\ 1\leq|l|\leq2.$
Our aim is to find a suitable real analytic symplectic coordinates transformation $\Phi$ such that $H\circ\Phi=N^*+P^*$ where $N^{*}$ has a similar form as $N$ and $P^*$ is analytic and globally of order 3. \\
\indent Actually, following the formulation of the KAM theorem given in \cite{Pos}(see also \cite{GT}), it is sufficient to verify that there exist $\xi\in \Pi$ with big Lebesgue measure which satisfy
 \begin{eqnarray}\label{ssmall}
|\la k,\omega_\nu(\xi)\ra+\la l,\Omega_\nu(\xi)\ra|\geq\frac{\la l\ra \alpha_\nu}{1+|k|^\tau},
 \end{eqnarray}
where $\omega_\nu(\xi), \Omega_\nu(\xi)$ are the frequencies in $\nu$-th KAM step.\\
\indent To obtain (\ref{ssmall}) in most cases we need some assumptions for $X_P$ such as \cite{Ku0} and \cite{Pos}. See Theorem 1 in \cite{Ku0} and its applications to 1d wave equation and 1d harmonic oscillator with a smooth nonlinearity of type $P=\frac12\int_\R\varphi(|u*\xi(x)|^2;a)dx)$($u*\xi$ is the convolution with a smooth real-valued function $\xi$, vanishing at infinity). In \cite{Pos} P\"oschel required a similar condition on $X_P$, i.e.
  \begin{eqnarray}\label{posac}
 X_P:\ \mathcal{P}^p\rightarrow\mathcal{P}^{\bar{p}},\qquad \bar{p}>p.
  \end{eqnarray}
See \cite{Pos1} for its application to a nonlinear wave equation.\\
\indent The assumption (\ref{posac}) is smartly weakened in \cite{GT}. Using the T\"oplitz-Lipschitz techniques from Eliasson and Kuksin in \cite{EK}, Gr\'{e}bert and Thomann assumed a weaker regularity on $P$, more clearly,
\begin{eqnarray*}
\Big\|\frac{\partial P}{\partial\omega_j}\Big\|^*_{D(s,r)}  \leq  \frac{r}{j^{\beta_1}}\langle P
\rangle^*_{r,D(s,r)},\qquad
\Big\|\frac{\partial^{2}P}{\partial\omega_j\partial\omega_l}\Big\|^*_{D(s,r)}  \leq  \frac{1}{(jl)^{\beta_1}}\langle P\rangle^*_{r,D(s,r)},
\end{eqnarray*}
for all  $j,l\geq 1$  and $ \omega_j=z_j,\overline{z}_j,$
where $\beta_1>0$ and $\|\cdot\|^*_{D(s,r)} $ stands for either $\|\cdot\|_{D(s,r)} $ or $\|\cdot\|^{\mathfrak{L}}_{D(s,r)} $. To recover this assumption at each step they noticed that $F$ satisfied an even better estimate, i.e.
\begin{eqnarray*}
\Big\|\frac{\partial F}{\partial\omega_j}\Big\|^*_{D(s,r)}  \leq  \frac{r}{j^{\beta_1+1}}\langle F
\rangle^{+,*}_{r,D(s,r)},\qquad
\Big\|\frac{\partial^{2}F}{\partial\omega_j\partial\omega_l}\Big\|^*_{D(s,r)}  \leq  \frac{\langle F\rangle^{+,*}_{r,D(s,r)} }{(jl)^{\beta_1}(1+|j-l|)}
\end{eqnarray*}
for all  $j,l\geq 1$  and $ \omega_j=z_j,\overline{z}_j.$\\
\indent In this paper we further weaken the assumptions on $P$, which satisfies the logarithmic decay, i.e.
\begin{eqnarray*}
\Big\|\frac{\partial P}{\partial\omega_j}\Big\|_{D(s,r)}^{*}  \leq  \frac{r}{(1+\ln j)^{\beta}}\langle P
\rangle_{r,D(s,r)}^{*},\qquad
\Big\|\frac{\partial^{2}P}{\partial\omega_j\partial\omega_l}\Big\|_{D(s,r)}^{*}  \leq  \frac{\langle P\rangle_{r,D(s,r)}^{*} }{(1+\ln j)^{\beta}(1+\ln l)^{\beta}}
\end{eqnarray*}
for all  $j,l\geq 1$  and $ \omega_j=z_j,\overline{z}_j.$   The index $\beta\geq2(n+2)$ is apparently different from $\beta_1>0$ in \cite{GT}, which we will explain in the following. As in \cite{GT}  we  obtain a better estimate for $F$, i.e.
\begin{eqnarray*}
\Big\|\frac{\partial F}{\partial\omega_j}\Big\|^*_{D(s,r)}  \leq \frac{r}{j(1+\ln j)^{\beta}}\langle F
\rangle^{+,*}_{r,D(s,r)},\quad
\Big\|\frac{\partial^{2}F}{\partial\omega_j\partial\omega_l}\Big\|^*_{D(s,r)}  \leq  \frac{\langle F\rangle^{+,*}_{r,D(s,r)} }{(1+\ln j)^{\beta}(1+\ln l)^{\beta}(1+|j-l|)}
\end{eqnarray*}
for all  $j,l\geq 1$  and $ \omega_j=z_j,\overline{z}_j.$ \\
\indent The shift in normal frequencies in the next step $\Omega^+_j(\xi)=\Omega_j(\xi)+\widehat{\Omega}_j(\xi)$ satisfies  much weaker estimate
  \begin{eqnarray}\label{oomega}
  |\widehat{\Omega}_j|\preceq\alpha(1+\ln j)^{-2\beta},\ j\geq 1,
     \end{eqnarray}
 comparing with the corresponding one in \cite{GT}, which is
\begin{eqnarray}\label{oomega1}
|\widehat{\Omega}_j|\preceq\alpha j^{-2\beta},\ j\geq 1.
\end{eqnarray}
Fact proved that the weaker estimate (\ref{oomega}) brought new  troubles  in measure estimates.  The consequence is that we need a new small divisor condition
 \begin{eqnarray}\label{ssmall0}
|\la k,\omega(\xi)\ra+\la l,\Omega(\xi)\ra|\geq\frac{\la l\ra \alpha}{\exp(|k|^{\tau/\beta})}, \qquad \beta\geq 2\tau\geq 2(n+2),
 \end{eqnarray}
in the KAM proof, which will be clear in the following. \\
\indent For simplicity we suppose $\omega(\xi)=\xi$ and $\Omega_j(\xi)=2j-1+O(\alpha(1+\ln j)^{-2\beta})(j\geq1)$.  Our main trouble is to estimate the set
 \begin{eqnarray*}\label{ssmall1}
\bigcup_{\substack{k,j, b\\ 1\leq b\leq c|k|}}\Big\{\xi\in\Pi:\ |f_{k,jb}(\xi) |<\frac{\alpha b}{\Delta_1(|k|)}\Big\},
 \end{eqnarray*}
where $$f_{k,jb}(\xi): = \la k,\omega(\xi)\ra+\Omega_i(\xi)-\Omega_j(\xi)=\la k,\xi\ra+2b+O\big( {\alpha}{(1+\ln j)^{-2\beta}}\big) ,$$
$i=j+b$ and $\Delta_1(\cdot)$ will be chosen later.
From a straightforward computation(see \ref{oomega}),  if $j\geq j_0=\exp\{(c\Delta(|k|))^{\frac{1}{2\beta}}\}$, then
 $|f_{k,jb}(\xi) |\geq {\alpha b}{\Delta^{-1}(|k|)}\geq {\alpha b}{\Delta^{-1}_1(|k|)}$, where $\Delta_1(|k|)\geq \Delta(|k|)$ and $|\la k,\xi\ra+2b|\geq  {2\alpha b}{\Delta^{-1}(|k|)}$.
For the rest,
 \begin{eqnarray}\label{ssmall2}
Meas\Big( \bigcup_{\substack{1\leq j\leq j_0\\1\leq b\leq c|k|}}\big\{\xi\in\Pi:\ |f_{k,jb}(\xi) |<\frac{\alpha b}{\Delta_1(|k|)}\big\}  \Big)\preceq  \frac{\alpha |k|^2}{\Delta_1(|k|)}\cdot \exp\big\{(c\Delta(|k|))^{\frac{1}{2\beta}}\big\}.
 \end{eqnarray}
(\ref{ssmall2}) explains (\ref{ssmall0}) since we set $\Delta(l)=l^\tau$,  and thus  $\Delta_1(l)=\exp(l^{\tau/\beta})$. The index $\beta\geq 2(n+2)$ comes from
 Lemma \ref{XFesti}($\beta>\tau$) and $\beta=\iota \tau\geq 2\tau\geq 2(n+2)$(see Theorem \ref{Meastheorem} for $\tau\geq n+2$ ). In the KAM proof we set $\iota\geq 2$ for simplicity. \\
\indent We remark that the structure (\ref{oomega}) or (\ref{oomega1}) is not necessary for some evolution equations. Based on the work \cite{EK}, Berti, Biasco and Procesi in \cite{BBP1,BBP2}  found a remarkable structure for $\Omega^+_j(\xi)=\Omega_j(\xi)+\widehat{\Omega}_j(\xi)$ in 1d derivative wave equation where $\widehat{\Omega}_j(\xi)=a_+(\xi)+O(\frac1j)(j\gg1)$ and $a_+(\xi)$  is independent of $j$(similar for $j<0$). But we don't know whether there exists similar or weaker structure for 1d harmonic oscillator, which is  a potential way to solve Eliasson's problem.\\
\indent To finish the proof of Theorem \ref{maintheorem}  we will follow the scheme given by Eliasson and Kuksin in \cite{EK0}.  As in \cite{EK0} the equation (\ref{maineq}) is rewritten into an autonomous Hamiltonian system in an extended phase space
$\mathcal{P}^2:=\T^n\times\R^n\times\ell^{2,2}\times \ell^{2,2}$,
 with the Hamiltonian function $H=N+\varepsilon P,$ where
 \begin{eqnarray*}
N:=N(\omega) = \sum_{
1\leq j\leq n}
\omega_j  y_j + \sum_{j\geq 1}
(2j-1)z_j\bar{z}_j.
 \end{eqnarray*}
 and $$ {P}(\theta,z,\bar{z})=\int_{\R}V(x,\theta;\omega)(\sum_{j\geq 1}z_jh_j)(\sum_{j\geq 1}\bar{z}_jh_j)dx$$ is quadratic in $(z,\bar{z})$.
 Here the external  parameters are the frequencies $\omega=(\omega_j)_{1\leq j\leq n}\in \Pi:=[0, 2\pi)^n$ and the normal frequencies $\Omega_j=2j-1$ are independent of $\omega$.
 To apply Theorem \ref{KAMtheorem} into the above Hamiltonian we need to check all the assumptions in Theorem \ref{KAMtheorem} are satisfied, and thus finish the proof of Theorem \ref{maintheorem}. In fact we need an improved estimate on $h_n(x)$. For simplicity we define the weighted $L^2$ norm of $h_n(x)$,
\begin{eqnarray*}
|||h_n(x)|||=\Big(\int_{\R}\frac{h_n^2(x)}{(1+\ln(1+x^2))^{2\delta_1}}dx\Big)^{\frac{1}{2}},
\end{eqnarray*}
with $\delta_1>0$.
\begin{lemma}\label{Indecaysection1}
Suppose $h_n(x)$ satisfies (\ref{eigensection1}), then for any $n\geq 1$,
$$|||h_n(x)|||\leq \frac{C_{\delta_1}}{(1+\ln n)^{\delta_1}}, $$
where $C_{\delta_1}$ is a constant depending on $\delta_1$ only.

\end{lemma}

\indent In the end of this section we give a fast description on  the recent development in KAM theory. For the KAM results with bounded perturbations see \cite{GY1,KLiang, KP, Ku0, LZ, LY1} for 1d-NLS. For high dimensional NLS see the milestone work by Eliasson and Kuksin \cite{EK}, where they  found and defined a T\"oplitz-Lipschitz property and used it to control the shift of the normal frequencies. See \cite{GXY} and \cite{PX} for recent development in nd-NLS. For nd-beam equations  see \cite{GY2} and \cite{GY3} for the nonlinearity $g(u)$ and see \cite{EGK1} and \cite{EGK2} for more general nonlinearity $g(x,u)$. Adapting the technics in \cite{EK} and \cite{GT}, Gr\'{e}bert and Paturel built a KAM for the Klein  Gordon equation on $S^d$ in \cite{GP}.  We remark that an earlier results  for NLW and NLS on compact Lie groups via Nash Moser technics can be found in \cite{BCP} (see also \cite{BP}), in which Berti, Corsi and Procesi proved the existence of quasi-periodic solutions without linear stability.  \\
\indent For the unbounded perturbations, the first KAM results have been obtained by Kuksin \cite{Ku1}-\cite{Ku2} for KdV with analytic perturbations(also see Kappeler  and
P\"oschel \cite{KaPo}).  See  \cite{LiuYuan} and \cite{ZGY} for recent progress for 1d derivative nonlinear Schr\"odinger equation(DNLS), Benjiamin-Ono equation and the reversible DNLS equation.

\noindent {\it Plan of the proof of Theorem \ref{maintheorem}}.  In Sect. \ref{S2} we give an abstract KAM
theorem(Thm. \ref{KAMtheorem}) which will be used to prove Thm. \ref{KAMapp}, a Hamiltonian formulation of Thm. \ref{maintheorem}, in Sect. \ref{S4}. In Sect. \ref{S3} we prove the weighted $L^2$ estimates on  the Hermite basis. The KAM proof of Thm. \ref{KAMtheorem} is deferred in Sect. \ref{S5} while the measure estimates are presented  in Sect. \ref{S6}. We put two technical inequalities in the Appendix.

\noindent {\it Notations}.   {$\la\cdot ,\cdot \ra$ is the standard scalar product in $\R^d$ or $\R^{\infty}$. $\|\cdot \|$ is an operator-norm or $\ell^2$-norm. $|\cdot |$ will in general denote a supremum norm with a notable exception. For $l=(l_1,l_2,\dots,l_k,\dots)\in \mathbb{Z}^\infty$ so that only a finite number of coordinates are
 nonzero, we denote by $|l|=\sum\limits_{j=1}^{\infty}|l_j|$  its length, $\langle l \rangle =\max\{1,\vert
 \sum\limits_{j=1}^{\infty}jl_j\vert\}$.  We use the notations $\Z_{+}=\{1,2,\cdots\}$ and $\N=\{0,1,2,\cdots\}$.\\
\indent The notation $``\preceq"$ used below  means $\leq $ modulo a multiplicative constant that, unless otherwise specified, depends on $n$ only.\\
\indent  We set $\mathcal{Z}=\{(k,l)\neq 0,\vert l \vert \leq 2 \} \subset
 \mathbb{Z}^n\times \mathbb{Z}^{\infty}$. Denote by $\Delta_{\xi
\eta} $ the difference operator in the variable $\xi$,
$ \Delta_{\xi\eta}f=f(\cdot,\xi)-f(\cdot,\eta)$, where $f$ is a real function. \\
\indent We denote
$\ell_{\C}^{2,p}$ the Hilbert space of all complex sequences $w=(w_j)_{j\geq 1}$
with
$\|w\|_{p}^2 = \sum\limits_{j\geq 1} j^p|w_j|^2<\infty$. We denote $\ell_\infty^{q}$  the space of all real(complex) sequences with finite norm
$|w|_{\ell_\infty^{q}}=\sup\limits_{j\geq 1}|w_j|j^q<\infty$, where  $q\in \R$.\\
\indent Following \cite{Pos}, we use $\|\cdot\|^{*}$(respectively $\la \cdot \ra^{*}$) stands either for $\|\cdot \|$ or $\|\cdot \|^{\mathfrak{L}}$(respectively $\la\cdot \ra$ or $\la \cdot \ra^{\mathfrak{L}}$) and $\|\cdot \|^{\lambda}$ stands for $\|\cdot \|+\lambda\|\cdot \|^{\mathfrak{L}}$. The notation Meas stands for the Lebesgue measure in $\R^n$. }

\section{KAM Theorem}\label{S2}
\subsection{KAM theorem}
\noindent Following the exposition in \cite{GT},  \cite{KLiang} and \cite{Pos}, we
consider small perturbations of a family of infinite-dimensional integrable Hamiltonians $N(y,u,v;\xi)$
with parameter $\xi$ in the normal form
 \begin{eqnarray}\label{unperturbedhamiltonian}
N = \sum_{
1\leq j\leq n}
\omega_j (\xi)y_j +\frac{1}{2}\sum_{j\geq 1}
\Omega_j(\xi)(u_j^2+v_j^2),
 \end{eqnarray}
on the phase space
$\mathcal{M}^p : = \T^n\times \R^n\times \ell^{2,p} \times \ell^{2,p}$
with coordinates $(\theta,y,u,v)$. The `internal' frequencies, $\omega=(\omega_j)_{1\leq j\leq n}$, as well as the `external' ones, $\Omega=(\Omega_j)_{j\geq 1}$, are real valued and depend on the parameter $\xi\in \Pi\subset \R^n$ and $\Pi$ is a compact subset of $\R^n$ of positive Lebesgue measure. The symplectic structure on $\mathcal{M}^p$ is the standard one given by
$\sum\limits_{1\leq j\leq n}d\theta_j\wedge dy_j+\sum\limits_{j\geq 1}du_j\wedge dv_j$. \\
\indent The Hamiltonian equations for motion of $N$ are therefore
$$\dot{\theta}=\omega(\xi), \qquad \dot{y}=0, \qquad \dot{u}=\Omega(\xi)v, \qquad \dot{v}=-\Omega(\xi)u,$$
where for any $j\geq 1$, $(\Omega(\xi)u)_j=\Omega_j u_j$. Hence, for any parameter $\xi\in \Pi$, on the n-dimensional invariant torus,
$$\T_0=\T^n\times \{0\}\times \{0\}\times\{0\},$$
the flow is rotational with internal frequencies $\omega(\xi)$. In the normal space, described by the $(u,v)$ coordinates, we have an elliptic equilibrium at the origin, whose frequencies are $\Omega(\xi)=(\Omega_j)_{j\geq 1}$. Hence, for any $\xi\in \Pi$, $\T_0$ is an invariant, rotational, linearly stable torus for the Hamiltonian $N$.\\
 \indent Our aim is to prove the persistence of this torus under small perturbations $N+P$ of the integrable Hamiltonian $N$ for a large Cantor set of parameter values $\xi$. To this end we make assumptions on the frequencies of the unperturbed Hamiltonian $N$ and on the perturbation $P$.\\
  \indent We need some notations for simplification. In the sequel, we use the distance
\begin{eqnarray*}
\| \Omega-\Omega'\|_{2\beta,\Pi}=\sup_{\xi\in\Pi}\sup_{j\geq1}(1+\ln j)^{2\beta}\vert\Omega_j(\xi)
-\Omega_j'(\xi)\vert,
\end{eqnarray*}
and the semi-norm,
$$ \|\Omega\|_{2\beta,\Pi}^{\mathfrak{L}}=\sup_{\substack{\xi,\eta\in\Pi\\\xi\neq\eta}}\sup_{j\geq1}
\frac{(1+\ln j)^{2\beta}\vert\triangle_{\xi\eta}\Omega_j\vert}{\vert\xi-\eta\vert}.
$$
\noindent{\bf Assumption $\mathcal{A}$} (Frequencies): \\
\textbf{(A1)} The map $\xi\mapsto\omega(\xi)$ between $\Pi$ and its image $\omega(\Pi)$ is a homeomorphism which, together with its inverse,  is Lipschitz continuous. \\
\textbf{(A2)} There exists a real sequence $(\overline{\Omega}_j)_{j\geq 1},$ independent of $\xi\in \Pi,$ of the form $\overline{\Omega}_j=a_{1}j+a_2$ with $a_1,a_2\in\R$ and $a_1\neq 0$, so that $ \xi\longmapsto (\Omega_j-\overline{\Omega}_j)_{j\geq 1}$ is a Lipschitz continuous map on $\Pi$ with values in $\ell_\infty^{-\delta}(\delta<0)$. More clearly, for $\xi\in \Pi,$ $|\Omega-\overline{\Omega}|_{\ell_\infty^{-\delta}}\leq M_1$ with $M_1>0.$\\
\textbf{(A3)} For all $(k,l)\in \mathcal{Z}$,
$$ Meas(\{\xi:k\cdot\omega(\xi)+l\cdot\Omega(\xi)=0\})=0.$$
and for all $\xi\in \Pi$,
$$ l\cdot\Omega(\xi)\neq0,\qquad \forall\  1\leq\vert l \vert \leq 2.$$
\indent The second set of assumptions concerns the perturbing Hamiltonian $P$ and its vector field, $X_{P}=
(\partial_yP,-\partial_{\theta}P, \partial_{v}P, -\partial_uP)$. We use the notation $i_{\xi} X_{P} $ for $X_{P}$ evaluated at $\xi$. Finally, we denote by $\mathcal{M}_{\C}^n$ the complexification of the phase space $\mathcal{M}^{p}$, $\mathcal{M}_{\C}^p = (\C/2\pi \Z)^n \times \C^n \times \ell^{2,p}_{\C}\times \ell^{2,p}_{\C}$.  Note that at each point of $\mathcal{M}_{\C}^p$, the tangent space is given by
$$\mathcal{P}_{\C}^{p} : = \C^{n}\times \C^{n}\times \ell_{\C}^{2,p}\times \ell_{\C}^{2,p}. $$
To state the assumptions about the perturbation we need to introduce some domains and norms.
For $s,r>0$, we define the complex neighborhood of $\T_0-$neighborhoods
$$
D(s,r)=\{|\Im \theta|<s\}\times \{|y|<r^2\}\times \{\|u\|_p+\|v\|_p<r\}\subset \mathcal{M}_{\C}^{p}.
$$
Here, for $a$ in $\R^n$ or $\C^n$, $|a|=\max_{j}|a_j|$ and $p\geq 2$. Let $r>0$, then for $W=(X,Y,U,V)$ in $\mathcal{P}_{\C}^p$ we denote
$$
\vert W\vert_r=\vert X \vert+\frac{1}{r^2}\vert Y\vert+\frac{1}{r}(\|U\|_p+\|V\|_p).$$
We then define the norms
$$\|P\|_{D(s,r)}:=\sup_{D(s,r)\times\Pi}\vert P\vert,\qquad\|P\|_{D(s,r)}^{\mathfrak{L}}:=\sup_{\substack{\xi,\eta\in\Pi\\ \xi\neq\eta}}\sup_{D(s,r)}\frac
{\vert\triangle_{\xi\eta} P \vert}{\vert \xi-\eta \vert},$$
 and we define the semi-norms
$$ \|X_P\|_{r,D(s,r)}:=\sup_{D(s,r)\times\Pi}\vert X_P\vert_r,\qquad \|X_P\|_{r,D(s,r)}^{\mathfrak{L}}:=\sup_{\substack{\xi,\eta\in\Pi\\ \xi\neq\eta}}\sup_{D(s,r)}\frac
{\vert\triangle_{\xi\eta} X_P\vert_r}{\vert \xi-\eta \vert}.$$
\indent In the sequel, we will often work in the complex coordinates
$z = \frac{1}{\sqrt{2}}(u-{\rm i} v),  z = \frac{1}{\sqrt{2}}(u+{\rm i}v)$.
Notice that this is not a canonical change of variables and in the variables
$(\theta, y, z, \bar{z}) \in \mathcal{M}_{\C}^p$ the symplectic structure reads
$\sum\limits_{1\leq j\leq  n}d \theta_j \wedge dy_j + {\rm i}\sum\limits_{j\geq 1} d z_j\wedge d\bar{z}_j $,
and the Hamiltonian in normal form is
 \begin{eqnarray}\label{newunperturbedhamiltonian}
N = \sum_{
1\leq j\leq n}
\omega_j (\xi)y_j +\sum_{j\geq 1}
\Omega_j(\xi)z_j\bar{z}_j,
 \end{eqnarray}

\noindent{\bf Assumption $\mathcal{B}$} (Perturbation):\\
\textbf{(B1)} We assume that there exist $s,r>0$ so that
$$ X_P:\qquad D(s,r)\times\Pi\longmapsto\mathcal{P}_{\C}^{p}.$$
Moreover $i_{\xi}X_P(\cdot,\xi)$ is analytic in $D(s,r)$ for each $\xi\in \Pi$.  $i_{w}P$ and $i_{w}X_{P}$ are uniformly Lipschitz on $\Pi$ for each $w\in D(s,r)$. \\
\indent Similar as \cite{GT}, we denote $\Gamma_{r,D(s,r)}^{\beta}$ as the following:
  Let $\beta>0$, we say that $P\in\Gamma_{r,D(s,r)}^{\beta}$ if $\langle P\rangle_{r,D(s,r)}+\langle P\rangle_{r,D(s,r)}^{\mathfrak{L}}<\infty$ where the norm $\langle\cdot\rangle_{r,D(s,r)}$ is defined by the conditions
\begin{eqnarray*}
\|P\|_{D(s,r)}&\leq& r^2\langle P \rangle_{r,D(s,r)},\\
\max_{1\leq j\leq n}\Big\|\frac{\partial P}{\partial y_j}\Big\|_{D(s,r)} &\leq& \langle P \rangle_{r,D(s,r)},\\
\Big\|\frac{\partial P}{\partial\omega_j}\Big\|_{D(s,r)} &\leq& \frac{r}{(1+\ln j)^{\beta}}\langle P
\rangle_{r,D(s,r)},\qquad \forall\ j\geq1\quad and \quad \omega_j=z_j,\overline{z}_j,\\
\Big\|\frac{\partial^{2}P}{\partial\omega_j\partial\omega_l}\Big\|_{D(s,r)} &\leq& \frac{1}{(1+\ln j)^{\beta}(1+\ln l)^{\beta}}\langle P\rangle_{r,D(s,r)},\qquad\forall\ j,l\geq 1\quad and \quad\omega_j=z_j,\overline{z}_j,
\end{eqnarray*}
and the semi-norm $\langle\cdot\rangle_{r,D(s,r)}^{\mathfrak{L}}$ is defined by the conditions
\begin{eqnarray*}
\|P\|_{D(s,r)}^{\mathfrak{L}}&\leq& r^2\langle P \rangle_{r,D(s,r)}^{\mathfrak{L}},\\
\max_{1\leq j\leq n}\Big\|\frac{\partial P}{\partial y_j}\Big\|_{D(s,r)}^{\mathfrak{L}} &\leq& \langle P \rangle_{r,D(s,r)}^{\mathfrak{L}},\\
\Big\|\frac{\partial P}{\partial\omega_j}\Big\|_{D(s,r)}^{\mathfrak{L}} &\leq& \frac{r}{(1+\ln j)^{\beta}}\langle P
\rangle_{r,D(s,r)}^{\mathfrak{L}},\qquad \forall\ j\geq1\quad and \quad \omega_j=z_j,\overline{z}_j,\\
\Big\|\frac{\partial^{2}P}{\partial\omega_j\partial\omega_l}\Big\|_{D(s,r)}^{\mathfrak{L}} &\leq& \frac{1}{(1+\ln j)^{\beta}(1+\ln l)^{\beta}}\langle P\rangle_{r,D(s,r)}^{\mathfrak{L}},\qquad\forall\ j,l\geq 1\quad and \quad\omega_j=z_j,\overline{z}_j.
\end{eqnarray*}
\textbf{(B2) }$P\in\Gamma_{r,D(s,r)}^{\beta}$ for some $\beta = \iota\tau\geq \iota(n+2)$ where $\iota\geq2$.\\
\begin{remark}
In the application to 1d quantum harmonic oscillator we will choose $\beta\geq 2(n+2)$, which is NOT the best choice.  But we have no intent to obtain the optimal one for $\beta$.
\end{remark}
We set
$
|\omega|_{\Pi}\leq M, \ | \omega^{-1}| _{\Pi}^{\mathfrak{L}}\leq L
$
and
\begin{eqnarray}\label{definitionM}
| \omega| _{\Pi}^{\mathfrak{L}}+\|\Omega\|_{2\beta,\Pi}^{\mathfrak{L}}\leq M,
\end{eqnarray}
where $ |\omega| _{\Pi}^{\mathfrak{L}}=\sup\limits_{\substack{\xi,\eta\in\Pi\\ \xi\neq\eta}}\max\limits_{1\leq k\leq n}\frac{\vert\triangle_{\xi\eta}\omega_k\vert}{\vert\xi-\eta\vert}$.\\
\begin{Theorem}\label{KAMtheorem}
({\rm KAM})Suppose that $N$ is a family of Hamiltonian of the form (\ref{newunperturbedhamiltonian}) defined on the phase space $\mathcal{M}^p$ with $p\geq 2$ depending on parameters $\xi\in\Pi$ so that Assumption $\mathcal{A}$ is  satisfied. Then there exist $\gamma>0$ and $s>0$ so that every perturbation $H=N+P$ of $N$ which satisfies Assumption $\mathcal{B}$ and the smallness condition
$$
\varepsilon=(\|X_P\|_{r,D(s,r)}+\langle P\rangle_{r,D(s,r)})+
\frac{\alpha}{M}(\|X_P\|_{r,D(s,r)}^{\mathfrak{L}}+\langle P\rangle_{r,D(s,r)}^{\mathfrak{L}})\leq\gamma\alpha^5,
$$
for some $r>0$ and $0<\alpha\leq1$, the following holds.\\
There exist\\
(i) a Cantor set $\Pi_{\alpha}\subset\Pi$ with $Meas(\Pi\backslash\Pi_{\alpha})\mapsto0$ as $\alpha\mapsto0$;\\
(ii) a Lipschitz family of real analytic,symplectic coordinates transformations
$$ \Phi:D(s/2,r/2)\times\Pi_{\alpha}\longmapsto D(s,r);$$
(iii) a Lipschitz family of new normal form
$$ N^{*}=\sum_{j=1}^{n}\omega_j^{*}(\xi)y_j+\sum_{j\geq1}\Omega_j^{*}(\xi)z_j\bar{z}_j$$
defined on $D(s/2,r/2)\times\Pi_{\alpha}$ such that
$$ H\circ\Phi=N^{*}+P^{*},$$
where $P^{*}$ is analytic on $D(s/2,r/2)$ and globally of order 3 at $\T_0$. That is the Taylor expansion of $P^{*}$ only contains monomials $y^mz^q\bar{z}^{\bar{q}}$ with $2\vert m \vert+\vert q+\bar{q}\vert\geq 3$. Moreover each symplectic coordinates transformation is close to the identity
\begin{eqnarray}\label{thmestimate1}
\|\Phi-Id\|_{r,D(s/2,r/2)}\leq c\varepsilon^{1/2},
\end{eqnarray}
and the new frequencies are close to the original ones
\begin{eqnarray}\label{thmestimate2}
 |\omega^{*}-\omega|_{\Pi_{\alpha}}+\|\Omega^{*}-\Omega\|_{2\beta,\Pi_{\alpha}}
\leq c\varepsilon,
\end{eqnarray}
and the new frequencies satisfy a non resonance condition
\begin{eqnarray}\label{thmestimate3}
 |k\cdot\omega^{*}(\xi)+l\cdot\Omega^{*}(\xi)|\geq\frac{\alpha}{2}\cdot\frac
{\langle l \rangle}{exp(|k|^{1/{\iota}})},\qquad\ \iota\geq2,\  ( k,l ) \in \mathcal{Z},\ \xi\in\Pi_{\alpha}.
\end{eqnarray}
\end{Theorem}
 {\begin{remark}
As a consequence, for each $\xi\in\Pi_{\alpha}$ the torus $\Phi(\T_0)$ is  invariant under the flow of the perturbed Hamiltonian $H=N+P$ and all these tori are linearly stable.
\end{remark}}

\section{Estimates on eigenfunctions in a weighted $L^2$ norm}\label{S3}
\noindent In this section we will prove Lemma \ref{Indecaysection1}.  A well-known fact is that
$h_n(x)={(n! 2^n \pi^{\frac12})^{-\frac12}}{e^{-\frac12 x^2} H_n(x)}$ where $H_n(x)$ is the Hermite polynomial of degree $n$ and $h_n(x)$ is an even or odd function of $x$ according to whether $n$ is odd or even(see Titchmarsh \cite{T1}).  The proof of Lemma \ref{Indecaysection1} is based upon Langer's turning point theory as presented in Chapter 22.27 of \cite{T2}(see \cite{Yajima}).
 {For simplicity we define the weighted $L^2$ norms of $h_n(x)$ on $\R$ and $\R_{\pm}$, which are
$$
|||h_n(x)|||=\Big(\int_{\R}\frac{h_n^2(x)}{(1+\ln(1+x^2))^{2\delta_1}}dx\Big)^{\frac{1}{2}},
$$
and
$$|||h_n(x)|||_{\pm}=\Big(\int_{\R_\pm}\frac{h_n^2(x)}{(1+\ln(1+x^2))^{2\delta_1}}dx\Big)^{\frac{1}{2}}
$$
with $\delta_1>0$.}\\
\indent From the symmetry
$|||h_n(x)|||^2= 2|||h_n(x)|||^2_{+}$,
and thus we only need to estimate $|||h_n(x)|||_{+}$. In the following we assume $n$ be sufficiently large.
As in \cite{Yajima},
\begin{eqnarray*}\label{gujidecay1}
h_n(x)&=&(\lambda_n-x^2)^{-\frac{1}{4}}(\frac{\pi\zeta}{2})^{\frac{1}{2}}H_{\frac{1}{3}}^{(1)}(\zeta)+(\lambda_n-x^2)^{-\frac{1}{4}}(\frac{\pi\zeta}{2})^{\frac{1}{2}}H_{\frac{1}{3}}^{(1)}(\zeta)\mathcal{O}(\frac{1}{\lambda_n})\\
&:=& \psi_1(x) +\psi_2(x),
\end{eqnarray*}
where $\zeta(x)=\int_X^x\sqrt{\lambda_n-t^2}dt$ with $X^2=\lambda_n$.
We only need to estimate $\psi_1(x)$ since the estimate for $\psi_2(x)$ is even better.
Let
\begin{eqnarray*}\label{}
Q(y)=\left\{
\begin{array}{cc}
-\int_y^1\sqrt{1-s^2}ds,&{\rm if}\ y<1,\\
{\rm i}\int_1^y\sqrt{s^2-1}ds, &{\rm if}\ y>1.
\end{array}
\right.
\end{eqnarray*}
We have $\zeta(x)=\lambda_n Q(y)$.
By Lemma 2.2 in \cite{Yajima} it holds that for any $K>1$
\begin{eqnarray*}\label{}
Q(y)&\sim&-(1-y)^{\frac{3}{2}},\ {\rm for}\ 0\leq y\leq1,\\
-{\rm i}Q(y)&\sim&(y-1)^{\frac{3}{2}},  \quad{\rm for}\  1\leq y\leq K,\\
-{\rm i}Q(y)&\sim & y^2,\ \quad \quad \quad \  {\rm for}\  y\geq K.
\end{eqnarray*}
 Recall that $H_{\frac{1}{3}}^{(1)}(\zeta)$ satisfies the following(\cite{T1}):\\
(1) when $\zeta=-z<0,$ $H_{\frac{1}{3}}^{(1)}(\zeta)=\frac{2}{\sqrt{3}}e^{-\frac{1}{6}\pi {\rm i}}(J_{\frac{1}{3}}(z)+J_{-\frac{1}{3}}(z))$ and
\begin{eqnarray}\label{realzeta}
\zeta^{\frac12}H_{\frac{1}{3}}^{(1)}(\zeta)=\left\{
\begin{array}{cc}
2^{\frac32}\pi^{-\frac{1}{2}}e^{\frac{1}{3}\pi {\rm i}}(\cos(z-\pi/4)+\mathcal{O}(z^{-1})),&z\rightarrow\infty,\\
2^{\frac23}3^{-\frac12}{\Gamma(2/3)}^{-1}e^{\frac{1}{3}\pi {\rm i}}z^{\frac16}(1+\mathcal{O}(z)),&z\rightarrow 0,
\end{array}
\right.
\end{eqnarray}
(2) when $\zeta={\rm i}w$ and $w\geq 0$, $H_{\frac{1}{3}}^{(1)}(\zeta)=\frac{2}{\pi}e^{-\frac{2}{3}\pi {\rm i}}K_{\frac{1}{3}}(w)$ and
\begin{eqnarray}\label{imagezeta}
\zeta^{\frac12}H_{\frac{1}{3}}^{(1)}(\zeta)=\left\{
\begin{array}{cc}
\mathcal{O}(e^{-w}),&w\rightarrow\infty,\\
2^{\frac13}e^{-\frac{1}{6}\pi }\pi^{-1}\Gamma(1/3)w^{\frac16}+\mathcal{O}(w^{\frac32})),&w\rightarrow 0.
\end{array}
\right.
\end{eqnarray}
If $n$ is large, then
\begin{eqnarray*}\label{}
|||h_n|||^2_{+}&\leq & 2\int_{0}^{+\infty}\frac{|\psi_1(x)|^2}{(1+\ln(1+x^2))^{2\delta_1}}dx+2\int_{0}^{+\infty}\frac{|\psi_2(x)|^2}{(1+\ln(1+x^2))^{2\delta_1}}dx\\
&\leq  &C\int_{0}^{+\infty}\frac{|\zeta^{\frac12}H_{\frac{1}{3}}^{(1)}(\zeta)|^2}{|1-y^2|^{\frac{1}{2}}(1+\ln(1+y^2X^2))^{2\delta_1}}dy.
\end{eqnarray*}
Lemma \ref{Indecaysection1} is a direct corollary from the following lemma.
\begin{lemma}\label{predecay}
There exists $C>0$ such that for large $n$,
\begin{eqnarray}\label{l2norm}\int_{0}^{+\infty}\frac{|\zeta^{\frac12}H_{\frac{1}{3}}^{(1)}(\zeta)|^2}{|1-y^2|^{\frac{1}{2}}(1+\ln(1+y^2X^2))^{2\delta_1}}dy\leq\frac{C\cdot 2^{2\delta_1}}{(1+\ln n)^{2\delta_1}}.
\end{eqnarray}
\end{lemma}
\begin{proof}
  We split the   integral into three parts
\begin{eqnarray*}\Big(\int_{0}^{1}+\int_{1}^{K}+\int_{K}^{+\infty}\Big)\frac{|\zeta^{\frac12}H_{\frac{1}{3}}^{(1)}(\zeta)|^2}{|1-y^2|^{\frac{1}{2}}(1+\ln(1+y^2X^2))^{2\delta}}dy
=I_1+I_2+I_3
\end{eqnarray*}
and estimate them separately.\\
(1) When $0\leq y\leq1$, $\zeta=-z<0,$ we split the integral $I_1$ into two parts $I_1=I_{11}+I_{12}$. Applying the first relation of (\ref{realzeta}) to $I_{11}$ and the second to $I_{12}$,
\begin{eqnarray*}\label{}
I_{11}&\leq& \int_0^{1-X^{-\frac23}}\frac{C}{|1-y^2|^{\frac{1}{2}}(1+\ln(1+y^2X^2))^{2\delta_1}}dy\\
&\leq& \Big(\int_0^{X^{-\frac12}}+\int_{X^{-\frac12}}^1\Big)\frac{C}{|1-y^2|^{\frac{1}{2}}(1+\ln(1+y^2X^2))^{2\delta_1}}dy\\
&\leq& C\Big(X^{-\frac12}+\frac{1}{(1+\ln(1+X))^{2\delta_1}}\Big)\leq  \frac{C\cdot 2^{2\delta_1}}{(\ln n)^{2\delta_1}},
\end{eqnarray*}
and
\begin{eqnarray*}\label{}
I_{12}&\leq& C\int_{1-X^{-\frac23}}^1\frac{\zeta^{\frac{1}{3}}}{|1-y^2|^{\frac{1}{2}}(1+\ln(1+y^2X^2))^{2\delta_1}}dy\\
&\leq &\int_{1-X^{-\frac23}}^1\frac{CX^{\frac23}(1-y)^{\frac{1}{2}}}{|1-y^2|^{\frac{1}{2}}(1+\ln(1+y^2X^2))^{2\delta_1}}dy \leq  \frac{C\cdot 2^{2\delta_1}}{(\ln n)^{2\delta_1}}.
\end{eqnarray*}
(2) When $1\leq y\leq K$   and $w=-{\rm i}\zeta\geq 0$, we split the integral $I_2$ into two parts $I_2=I_{21}+I_{22}$. Applying the first relation of (\ref{imagezeta}) to $I_{21}$ and the second to $I_{22}$, we obtain
\begin{eqnarray*}\label{}
I_{21}&\leq& \int_{1+X^{-\frac23}}^K\frac{Ce^{-2X^2(y-1)^{\frac{3}{2}}}}{|1-y^2|^{\frac{1}{2}}(1+\ln(1+y^2X^2))^{2\delta_1}}dy\\
&\leq &\frac{C}{(\ln n)^{2\delta_1}}\int_{1+X^{-\frac23}}^{K}\frac{e^{-2X^2(y-1)^{\frac32}}}{(y-1)^{\frac12}}dy \leq   \frac{C{\lambda_n}^{-\frac23}}{(\ln n)^{2\delta_1}}\leq  \frac{C}{(\ln n)^{2\delta_1}},
\end{eqnarray*}
and
\begin{eqnarray*}\label{}
I_{22}&\leq& \int_{1 }^{1+X^{-\frac23}}\frac{CX^{\frac23}(y-1)^{\frac{1}{2}}}{|1-y^2|^{\frac{1}{2}}(1+\ln(1+y^2X^2))^{2\delta_1}}dy\\
&\leq&  \frac{C}{(1+\ln(1+X^2))^{2\delta_1}} \leq  \frac{C}{(\ln n)^{2\delta_1}}.
\end{eqnarray*}
(3) When $y\geq K$   and $w=-{\rm i}\zeta>0$, we apply  the first relation of (\ref{imagezeta}) to $I_{3}$,
\begin{eqnarray*}\label{}
I_{3}&\leq& \int^{\infty}_K\frac{Ce^{-2w}}{|1-y^2|^{\frac{1}{2}}(1+\ln(1+y^2X^2))^{2\delta_1}}dy\\
&\leq& \int^{\infty}_K\frac{Ce^{-2X^2y^2}}{y(1+\ln(1+y^2X^2))^{2\delta_1}}dy\\
&\leq& C \frac{1}{(1+\ln(1+X^2))^{2\delta_1}} \leq  \frac{C}{(1+\ln n)^{2\delta_1}}.
\end{eqnarray*}
Combining the above estimates together, we obtain (\ref{l2norm}).
\end{proof}

\section{Application to quantum harmonic oscillators}\label{S4}
In this section, we will apply Theorem \ref{KAMtheorem} to our model equation (\ref{maineq}) and prove the results stated in Section \ref{introduction}.
For readers' convenience, we rewrite the equation
\begin{eqnarray}\label{maineq1}
{\rm i}\partial_tu=-\partial_x^2 u+x^2u+\varepsilon V(x, \omega t;\omega)u,\ \ \ u=u(t,x),\ x\in\R,
 \end{eqnarray}
 where the potential $V: \R\times  \T^n \times \Pi\ni(x, \theta;\omega)\mapsto
\R$ is $C^3$ smooth  in all its variables and analytic in $\theta$.  For $\rho>0$ the function $V(x, \theta;\omega)$ analytically in $\theta$ 
extends to the domain
$\T_{\rho}^n$
as well as its gradient in $\omega$
and satisfies (\ref{VV1})--(\ref{VV3}) with $\beta\geq 2(n+2)$.\\
\indent In the following we will follow the scheme developed by  Eliasson and  Kuksin in \cite{EK0}. Expand $u$ and $\bar{u}$ on the Hermite basis $\{h_j\}_{j\geq1}$,
$u=\sum_{j\geq 1}z_jh_j,\ \ \bar{u}=\sum_{j\geq 1}\bar{z}_jh_j.$
And thus  (\ref{maineq1}) can be written as a nonautonomous Hamiltonian system
\begin{eqnarray}\label{hs01}
\left\{
\begin{array}{cc}
\displaystyle\dot{z}_j=-{\rm i}(2j-1)z_j-{\rm i}\varepsilon\frac{\partial}{\partial \bar{z}_j}\widetilde{P}(t,z,\bar{z}),\ j\geq 1,\\
\displaystyle\dot{\bar{z}}_j=\ \ {\rm i}(2j-1)\bar{z}_j+{\rm i}\varepsilon\frac{\partial}{\partial {z}_j}\widetilde{P}(t,z,\bar{z}),\   j\geq 1,
\end{array}
\right.
\end{eqnarray}
where $$\widetilde{P}(t,z,\bar{z})=\int_{\R}V(x,\omega t;\omega)(\sum_{j\geq 1}z_jh_j)(\sum_{j\geq 1}\bar{z}_jh_j)dx,$$ and $(z,\bar{z})\in\ell^{2,2}\times \ell^{2,2}.$ As \cite{EK0} and \cite{GT}, we write (\ref{hs01}) as an autonomous Hamiltonian system in an extended phase space
$\mathcal{P}^2:=\T^n\times\R^n\times\ell^{2,2}\times \ell^{2,2}$,
\begin{eqnarray}\label{hs02}
\left\{
\begin{array}{cc}
\displaystyle\dot{z}_j=-{\rm i}(2j-1)z_j-{\rm i}\varepsilon\frac{\partial}{\partial \bar{z}_j} {P}(\theta,z,\bar{z}),\ j\geq 1,\\
\displaystyle\dot{\bar{z}}_j=\ \ {\rm i}(2j-1)\bar{z}_j+{\rm i}\varepsilon\frac{\partial}{\partial {z}_j} {P}(\theta,z,\bar{z}),\   j\geq 1,\\
\displaystyle\dot{\theta}_j=\omega_j,\ \ \ \ \ \ \ \ \ \ \ \ \ \ \ \ \ \ \ \ \ \ \ \ \  j=1,2,\ldots,n,\\
\displaystyle\dot{y}_j=-\varepsilon\frac{\partial}{\partial \theta_j} {P}(\theta,z,\bar{z}),\ \ \ \ \ \ \ \ j=1,2,\ldots,n,
\end{array}
\right.
\end{eqnarray}
 with the Hamiltonian function $H=N+\varepsilon P,$ where
 \begin{eqnarray*}\label{NN}
N:=N(\omega) = \sum_{
1\leq j\leq n}
\omega_j  y_j + \sum_{j\geq 1}
(2j-1)z_j\bar{z}_j.
 \end{eqnarray*}
 and $$ {P}(\theta,z,\bar{z})=\int_{\R}V(x, \theta;\omega)(\sum_{j\geq 1}z_jh_j)(\sum_{j\geq 1}\bar{z}_jh_j)dx$$ is quadratic in $(z,\bar{z})$.
 {Here the external  parameters are the frequencies $\omega=(\omega_j)_{1\leq j\leq n}\in \Pi:=[0, 2\pi)^n$ and the normal frequencies $\Omega_j=2j-1$ are independent of $\omega$.}
We remark that the first three equations of (\ref{hs02}) are independent of $y$ and  equivalent to (\ref{hs01}). \\
\indent Similar as \cite{GT}, we have
\begin{Theorem}\label{KAMapp}
Assume that $V$ satisfies all the conditions in Theorem \ref{maintheorem} and $\beta\geq 2(n+2)$. Then there exists $\varepsilon_0$ such that for all $0 \leq\varepsilon<\varepsilon_0$ there exist\\
 (i)\ $\Pi_\varepsilon\subset [0, 2\pi)^n$ of positive measure and
$Meas(\Pi_\varepsilon)\rightarrow(2\pi)^n$ as $\varepsilon\rightarrow 0$; \\
 (ii)\ a Lipschitz family of real analytic, symplectic and linear coordinates transformation $\Phi:\ \Pi_\varepsilon\times\mathcal{P}^0\mapsto\mathcal{P}^0$ of the form
    \begin{eqnarray}\label{phifin}
\Phi_\omega(y,\theta,\zeta) =(y+\frac{1}{2}\zeta\cdot M_\omega(\theta)\zeta,\ \theta,\ L_\omega(\theta)\zeta)
 \end{eqnarray}
 where $\zeta=(z,\bar{z}),\  M_\omega(\theta)$ and $L_\omega(\theta)$ are linear bounded operators from $\ell^{2,p}\times \ell^{2,p}$ into itself for all $p\geq0$ and $L_\omega(\theta)$ is invertible;\\
 (iii)\ a Lipschitz family of new normal forms
  \begin{eqnarray*}\label{}
N^*(\omega) = \sum_{
1\leq j\leq n}
\omega_j  y_j + \sum_{j\geq 1}
\Omega^*_j(\omega)z_j\bar{z}_j;
 \end{eqnarray*}
 such that on $\Pi_\varepsilon\times\mathcal{P}^0$, $H\circ\Phi=N^*$.\\
\indent Moreover the new external frequencies are close to the original ones, $\|\Omega^*-\Omega\|_{2\beta,\Pi_\varepsilon}\leq c\varepsilon,$ and the new frequencies satisfy a nonresonant condition, i.e.
 $$ |k\cdot\omega+l\cdot\Omega^{*}(\omega)|\geq\frac{\alpha}{2}\cdot\frac
{\langle l \rangle}{exp(|k|^{1/{\iota}})},\qquad\ \iota\geq2,\  ( k,l ) \in \mathcal{Z},$$
for some $\alpha>0$ and $\omega\in\Pi_\varepsilon$.
 \end{Theorem}
 \begin{proof}
 As \cite{GT}, Assumption $\mathcal{A}$ is clear. We now check Assumption $\mathcal{B}$ holds. Firstly, we need to check that $(\frac{\partial P}{\partial z_k})_{k\geq1}\in\ell^{2,2}$ and $(\frac{\partial P_{\omega_j}}{\partial z_k})_{k\geq1}\in\ell^{2,2}$ with $j=1,\cdots, n$. \\
 \indent Note that $\frac{\partial P}{\partial z_k}=\int_{\R}V(x, \theta;\omega)\bar{u}h_kdx$, which is the $k$-th coefficient of the decomposition of $V(x,\theta;\omega)\bar{u}$ in the Hermite basis. It follows that $(\frac{\partial P}{\partial z_k})_{k\geq1}\in\ell^{2,2}$ if and only if $V(x, \theta;\omega)\bar{u}\in{\mathcal{H}}^2$.
 From $
|V| \leq C,
\ |\partial_xV |\leq C,
\ |\partial^2_{x}V | \leq   C
  $, $\bar{u}\in{\mathcal{H}}^2$ and a straightforward computation, we have $V(x, \theta;\omega)\bar{u}\in{\mathcal{H}}^2$. This implies that $(\frac{\partial P}{\partial z_k})_{k\geq1}\in\ell^{2,2}$. Similarly from $
|\partial_{\omega_j} V| \leq C,
\ |\partial_x(\partial_{\omega_j} V )| \leq C,
\ |\partial^2_{x}(\partial_{\omega_j} V) | \leq   C
  $ and $\bar{u}\in{\mathcal{H}}^2$ we obtain    $(\frac{\partial  P_{\omega_j}}{\partial z_k})_{k\geq1}\in\ell^{2,2}$ and thus (B1) is satisfied.\\
   \indent In the following we turn to (B2) in Assumption $\mathcal{B}$.
  {From (\ref{VV1}), Lemma \ref{Indecaysection1} and a straightforward computation,
  \begin{eqnarray*}\label{}
\Big\|\frac{\partial P}{\partial z_k}\Big\|_{D(s,r)}&=&\sup\limits_{D(s,r)}\Big|\int_{\R}V(x, \theta;\omega)\bar{u}h_kdx\Big|\\
&\preceq&  \Big(\int_{\R}\frac{|h_k|^2}{(1+\ln (1+x^2))^{2\beta}}dx\Big)^{\frac{1}{2}}\cdot \sup\limits_{D(s,r)}\Big(\int_{\R}|\bar{u}|^{2}dx\Big)^{\frac{1}{2}}\preceq \frac{C_\beta r}{(1+\ln k )^{\beta}}.
\end{eqnarray*}
Similarly,\begin{eqnarray*}\label{}
\Big\|\frac{\partial^2 P}{\partial z_k\partial \bar{z}_l}\Big\|_{D(s,r)}&=&\sup\limits_{D(s,r)}\Big|\int_{\R}V(x, \theta;\omega)h_kh_ldx\Big|
\preceq \frac{C_\beta}{(1+\ln k )^{\beta}(1+\ln l )^{\beta}}.
\end{eqnarray*}}
From the conditions (\ref{VV1}) - (\ref{VV3}) and a similar computation we obtain
  \begin{eqnarray*}\label{}
\Big\|\frac{\partial P}{\partial z_k}\Big\|^{\mathfrak{L}}_{D(s,r)}\preceq \frac{C_\beta r}{(1+\ln k )^{\beta}}\quad
{\rm{and}}\ \quad
 \Big\|\frac{\partial^2 P}{\partial z_k\partial \bar{z}_l}\Big\|^{\mathfrak{L}}_{D(s,r)}
\preceq \frac{C_\beta}{(1+\ln k )^{\beta}(1+\ln l )^{\beta}}.
\end{eqnarray*} It follows that $P\in\Gamma_{r,D(s,r)}^{\beta}$ with $s=\rho$ and $\beta\geq 2(n+2).$\\
 \indent For our application to Theorem \ref{maintheorem} we will choose $M=2\pi$, $\beta=\iota (n+2)$ with $\iota\geq 2.$ A straight computation shows that
 {\begin{eqnarray*}\label{}
&&\|X_{\varepsilon P}\|_{r,D(\rho,r)}+\langle \varepsilon P\rangle_{r,D(\rho,r)}+
\frac{\alpha}{2\pi}\big(\|X_{\varepsilon P}\|_{r,D(\rho,r)}^{\mathfrak{L}}+\langle {\varepsilon P}\rangle_{r,D(\rho,r)}^{\mathfrak{L}}\big)\\
&&\preceq\frac{c(\iota,n)\varepsilon}{\rho}(1+\alpha)\leq\frac{2c(\iota,n)\varepsilon}{\rho}\leq\gamma \alpha^5 ,\end{eqnarray*}
if we choose $\alpha=\varepsilon^{\frac{1}{10}}$ and $\varepsilon\leq\varepsilon_0:=(\frac{\gamma \rho}{2c(\iota,n)})^2$.\\}
 \indent Therefore Theorem \ref{KAMtheorem} applies with $p=2.$ Following \cite{GT} we have:\\
 (i)\ the symplectic coordinates transformation $\Phi$ is quadratic and thus it is defined on the whole phase space and have the form
    \begin{eqnarray*}\label{}
\Phi_\omega(y,\theta,\zeta) =(y+\frac{1}{2}\zeta\cdot M_\omega(\theta)\zeta,\ \theta,\ L_\omega(\theta)\zeta);
 \end{eqnarray*}
(ii)\ the new normal form still have the same frequencies vector $\omega$;\\
(iii)\ the new Hamiltonian reduces to the new normal form, i.e., $R^*=0$;\\
(iv)\ the symplectic coordinates transformation $\Phi_\omega$, which is defined by Theorem \ref{KAMapp} on each $\mathcal{P}^2$, extends to $\mathcal{P}^0:=\T^n\times\R^n\times\ell^{2,0}\times \ell^{2,0}$.\\
\indent  We complete the proof of Theorem \ref{KAMapp}. Meanwhile, Theorem \ref{maintheorem} follows directly from Theorem \ref{KAMapp}.
 \end{proof}
 \noindent{\bf{Proof of Corollary} \ref{coro1}.} See \cite{GT}.\qed\\
\noindent{\bf{Proof of Corollary} \ref{coro2}.} We follow the scheme developed in \cite{GT}. Firstly we write the solution $u(t,x)$ of (\ref{maineq1}) with initial datum $u_0(x)=\sum_{j\geq1}z_j(0)h_j(x)$ as $u(t,x)=\sum_{j\geq1}z_j(t)h_j(x)$ with
$$(z,\bar{z})(t)=L_\omega(\omega t)(z'(0)e^{-{\rm i}\Omega^*t},\bar{z}'(0)e^{{\rm i}\Omega^*t})$$
and $(z'(0),\bar{z}'(0) )=L^{-1}_\omega(0)(z (0),\bar{z} (0) )$. From the structure of $L_\omega(\theta)$, more clearly  $(L_\omega(\theta))_{jk}^{12}=(L_\omega(\theta))_{jk}^{21}=0$, we then have
$$u(t,x)=\sum_{j\geq1}\psi_j(\omega t,x)e^{-{\rm i}\Omega_j^*t},$$
where $\psi_j(\theta,x)=\sum_{l\geq1}(L_\omega(\theta))_{jk}^{11}z_j'(0)h_l(x)$. In particular the solutions are all almost periodic in time with a nonresonant frequencies vector $(\omega,\Omega^*).$ By a straight computation we can prove that $\psi_j(\omega t,x)e^{-{\rm i}\Omega_j^*t}$ solves (\ref{maineq1}) if and only if $k\cdot \omega+\Omega^*_j$ is an eigenvalue of $K$ with eigenfunction $\psi_j(\theta,x)e^{{\rm i}k\cdot\theta}$. This shows that the spectrum set of the Floquet operator $K$ equals  to $\{k\cdot \omega+\Omega^*_j|k\in\Z^n,\ j\geq 1\}$ and thus we complete the proof.\qed

\section{Proof of KAM Theorem}\label{S5}
\subsection{The linearized equation}
Let $H = N +P $ be a Hamiltonian, where
\begin{eqnarray}\label{1.2}
N = \sum_{
1\leq j\leq n}
\omega_j (\xi)y_j +\sum_{j\geq 1}
\Omega_j(\xi)z_j\bar{z}_j,
 \end{eqnarray} and $P$   satisfies  Assumption $\mathcal{B}$ in Sect. \ref{S2}. The aim in this section is to put $N+P$ into a new normal form $N_{+}+P_{+}$ such that $P_{+}$ is much smaller than $P$. To do this, we need to solve the homological equation
\begin{eqnarray}\label{homo}
\{F, N\} +  \widehat{N} = R,
\end{eqnarray}where $R$ is the second order Taylor approximation  of $P$,
\begin{eqnarray}\label{2.9}
R=\sum_{2|m|+|q+\bar{q}|\leq 2}\sum_{k\in \Z^n} R_{kmq\bar{q}}e^{ik\theta}y^m z^q \bar{z}^{\bar{q}}
\end{eqnarray}
with $R_{kmq\bar{q}}=P_{kmq\bar{q}}$,
and $F$ has a similar form as $R$,
\begin{eqnarray}\label{fpre}
F=\sum_{2|m|+|q+\bar{q}|\leq 2}\sum_{ k\in \Z^n} F_{kmq\bar{q}}e^{{\rm i}k\theta}y^m z^q \bar{z}^{\bar{q}}.
\end{eqnarray}
From \cite{Pos}, we have
\begin{lemma}\label{XFesti}
 Suppose   $|\omega|^{\mathfrak{L}}+\|\Omega\|^{\mathfrak{L}}_{2\beta,\Pi}\leq M$ uniformly on $\Pi$ and
$$|\la k,\omega(\xi)\ra +\la l,\Omega(\xi)\ra|\geq \frac{\la l\ra \alpha}{A_k},\ \ (k,l)\in\mathcal{Z},$$
where $\alpha>0$ and $A_k=e^{|k|^{\tau/\beta}}(\beta>\tau)$. Then the linearized equation $\{F, N\} +  \widehat{N} = R$ has a
solution $F,\  \widehat{N}$ satisfying $[F]=0,\  \widehat{N} = [R]=\sum\limits_{|m|+|q|=1}R_{0mq{q}}y^m z^q \bar{z}^{{q}},
 $ and
\begin{eqnarray*}
&&\|X_{\widehat{N}}\|_r\leq \|X_{R}\|_r\ , \hspace{1cm} \|X_{F}\|_{r,D(s-\sigma,r )}\leq \frac{c(n,\beta)e^{3(\frac{4}{\sigma})^{t_1}}}{\alpha}\|X_{R}\|_r\ ,\\
&&\|X_{\widehat{N}}\|^{\mathfrak{L}}_r\leq \|X_{R}\|^{\mathfrak{L}}_r\ ,\hspace{0.9cm} \|X_{F}\|^{\mathfrak{L}}_{r,D(s-\sigma,r )}\leq \frac{c(n,\beta)e^{3(\frac{4}{\sigma})^{t_1}}}{\alpha}(\|X_{R}\|^{\mathfrak{L}}_r+\frac{M}{\alpha}\|X_{R}\|_r)\
\end{eqnarray*}
for
$ 0< \sigma\leq s$,   $t_1=\frac{\tau}{\beta-\tau} $ and the short hand  $\|\cdot\|_r = \|\cdot\|_{r,D(s,r )}$ is used.
\end{lemma}
 \indent Introduce the space $\Gamma_{r,D(s,r)}^{\beta,+}\subset \Gamma_{r,D(s,r)}^{\beta}$ endowed with the norm $\la \cdot\ra_{r,D(s,r)}^{+}+\la \cdot\ra_{r,D(s,r)}^{+,\mathfrak{L}}$ defined by the following conditions:
\begin{eqnarray*}
\|F\|^{*}_{D(s,r)}&\leq&r^2\la F \ra_{r,D(s,r)}^{+,*},\ \ \ \ \ \max_{1\leq j \leq n}\Big\|\frac{\partial F}{\partial y_j}\Big\|^{*}_{D(s,r)}\leq \la F \ra_{r,D(s,r)}^{+,*},\\
\Big\|\frac{\partial F}{\partial w_j}\Big\|^{*}_{D(s,r)}&\leq&\frac{r}{j(1+\ln j)^{\beta}}\la F \ra_{r,D(s,r)}^{+,*},\ \ \forall j\geq 1\ {\rm and}\ w_j=z_j \ {\rm or}\ \bar{z}_j,\\
\Big\|\frac{\partial^2 F}{\partial w_j\partial w_l}\Big\|^{*}_{D(s,r)}&\leq&\frac{\la F \ra_{r,D(s,r)}^{+,*}}{(1+|j-l|)(1+\ln j)^{\beta}(1+\ln l)^{\beta}}, \ \ \forall j,l\geq 1\ {\rm and}\ w_j=z_j \ {\rm or}\ \bar{z}_j.
\end{eqnarray*}

\begin{lemma}\label{fsemi}
Assume that the frequencies satisfy
\begin{equation}\label{sdc}|\la k,\omega(\xi)\ra +\la l,\Omega(\xi)\ra|\geq \frac{\la l\ra \alpha}{A_k},\ \ (k,l)\in\mathcal{Z},
\end{equation}
where $\alpha>0$ and $A_k=e^{|k|^{\tau/\beta}}(\beta>\tau)$ uniformly on $\Pi$.  Let $F,\widehat{N}$ be given in the above lemma and
  $R\in\Gamma_{r,D(s,r)}^{\beta}$, then for any $0<\sigma<s$, we have $F\in\Gamma_{r,D(s-\sigma,r)}^{\beta,+}$, $\widehat{N}\in\Gamma_{r,D(s-\sigma,r)}^{\beta}$ such that
\begin{eqnarray*}
\la F \ra_{r,D(s-\sigma,r)}^{+}&\leq &\frac{c(n,\beta)e^{2(\frac{2}{\sigma})^{t_1}}}{\alpha}\la R \ra_{r,D(s,r)},\\
\la F \ra_{r,D(s-\sigma,r)}^{+,\mathfrak{L}}&\leq &\frac{c(n,\beta)e^{6(\frac{8}{\sigma})^{t_1}}}{\alpha}\Big(\frac{M}{\alpha}\la R \ra_{r,D(s,r)}+\la R \ra_{r,D(s,r)}^{\mathfrak{L}}\Big)
\end{eqnarray*}
with   $t_1=\frac{\tau}{\beta-\tau}$, and
\begin{eqnarray*}
\la \widehat{N} \ra_{r,D(s-\sigma,r)}\preceq \la R \ra_{r,D(s,r)},\quad\quad
\la \widehat{N} \ra_{r,D(s-\sigma,r)}^{\mathfrak{L}}\preceq\la R \ra_{r,D(s,r)}^{\mathfrak{L}}.
\end{eqnarray*}

\end{lemma}
\begin{proof}
Our aim is to solve the homological equation (\ref{homo}) and find a solution $F$ for it.
A straightforward  computation shows that the coefficients in (\ref{fpre}) are given by
\begin{eqnarray}\label{homo02}
 {\rm i}F_{kmq\bar{q}}=\left\{
\begin{array}{cc}
\displaystyle\frac{R_{kmq\bar{q}}}{k\cdot \omega+(q-\bar{q})\Omega},& {\rm if}\ |k|+|q-\bar{q}|\neq0,\\
\displaystyle0,& {\rm otherwise}.
\end{array}\right.
\end{eqnarray}
As \cite{GT}, in the following we will use the notation $q_j = (0, \cdots, 0, 1, 0, \cdots)$, where ``1" is
in the $j-$th position  and $q_{jl} = q_j + q_l $. The variables $z$ and $\bar{z}$ exactly play the same role,
therefore it is enough to study the derivatives in  $z$. We first show that
\begin{eqnarray*}
\la F \ra_{r,D(s-\sigma,r)}^{+}&\preceq &\frac{c(n,\beta)e^{2(\frac{2}{\sigma})^{t_1}}}{\alpha}\la R \ra_{r,D(s,r)}.
\end{eqnarray*}
From
$|R_{k0q_{jl}0}|\preceq \frac{ \la R \ra_{r,D(s,r)}e^{-|k|s}}{(1+\ln j)^\beta(1+\ln l)^\beta}$, (\ref{sdc}) and (\ref{homo02}),
$$|F_{k0q_{jl}0}|\preceq \frac{ A_k\la R \ra_{r,D(s,r)}e^{-|k|s}}{\alpha|j+l|(1+\ln j)^\beta(1+\ln l)^\beta} \preceq \frac{ A_k\la R \ra_{r,D(s,r)}e^{-|k|s}}{\alpha(1+|j-l|)(1+\ln j)^\beta(1+\ln l)^\beta}.$$
Therefore,
\begin{eqnarray}\label{f01}
\nonumber\Big\|\frac{\partial^2 F}{\partial z_j\partial z_l} \Big\|_{D(s-\sigma,r)}&\leq &\sum_{k\in \Z^n}|F_{k0q_{jl}0}|e^{|k|(s-\sigma)}\\
\nonumber&\preceq& \sum_{k\in \Z^n}\frac{ A_k\la R \ra_{r,D(s,r)}e^{-|k|\sigma}}{\alpha(1+|j-l|)(1+\ln j)^\beta(1+\ln l)^\beta}\\
&\preceq& \frac{c(n,\beta)e^{2(\frac{2}{\sigma})^{t_1}}}{\alpha(1+|j-l|)(1+\ln j)^\beta(1+\ln l)^\beta}\la R \ra_{r,D(s,r)}.
\end{eqnarray}
From Lemma \ref{basic2} and a similar computation, we have
\begin{eqnarray}\label{f02}
\Big\|\frac{\partial F}{\partial z_j} \Big\|_{D(s-\sigma,r)}\preceq  \frac{c(n,\beta)re^{2(\frac{2}{\sigma})^{t_1}}}{\alpha(1+j)(1+\ln j)^\beta}\la R \ra_{r,D(s,r)}.
\end{eqnarray}
\indent Similarly,  \begin{eqnarray}\label{f03}
\Big\|\frac{\partial F}{\partial y_j} \Big\|_{D(s-\sigma,r)}\preceq  \frac{c(n,\beta)e^{2(\frac{2}{\sigma})^{t_1}}}{\alpha}\la R \ra_{r,D(s,r)},
\end{eqnarray}
and
\begin{eqnarray}\label{f04}
\|F\|_{D(s-\sigma,r)}\preceq \frac{c(n,\beta)e^{2(\frac{2}{\sigma})^{t_1}}}{\alpha}r^2\la R \ra_{r,D(s,r)}.
\end{eqnarray}
 The above  estimates (\ref{f01}), (\ref{f02}), (\ref{f03}) and (\ref{f04}) show us  that
\begin{eqnarray*}
\la F\ra^+_{r,D(s-\sigma,r)}\preceq\frac{c(n,\beta)e^{2(\frac{2}{\sigma})^{t_1}}}{\alpha}\la R \ra_{r,D(s,r)}.
\end{eqnarray*}
\indent It remains to check the estimates with the Lipschitz semi-norms. As in \cite{Pos}, for
$|k| + |q_j-q_l | \neq 0$ define $\delta_{k, jl} = k \cdot \omega+ \Omega_j-\Omega_l $ and $\Delta=\Delta_{\xi\zeta}$.
Then we have
$${\rm i}\Delta{F}_{k0q_j\bar{q}_l}=\delta_{k,jl}^{-1}(\eta)\Delta{R}_{k0q_j\bar{q}_l}+{R}_{k0q_j\bar{q}_l}(\xi)\Delta\delta_{k,jl}^{-1},$$
and
$$-\Delta\delta_{k,jl}^{-1}=\frac{\la k,\Delta\omega\ra+\Delta\Omega_j-\Delta\Omega_l}{\delta_{k,jl}(\xi)\delta_{k,jl}(\zeta)}.$$
By the small divisor assumptions and a direct computation, we have
\begin{eqnarray*}
\frac{|\Delta{F}_{k0q_j\bar{q}_l}|}{|\xi-\eta|}&\preceq&\frac{A_k}{\alpha \la j-l\ra}\frac{|\Delta{R}_{k0q_j\bar{q}_l}|}{|\xi-\eta|}+\frac{M|k|A^2_k}{\alpha^2\la j-l\ra^2}|{R}_{k0q_j\bar{q}_l}|\\
&\preceq &\frac{|k|A_k^2}{\alpha \la j-l\ra}\Big(\frac{|\Delta{R}_{k0q_j\bar{q}_l}|}{|\xi-\eta|}+\frac{M}{\alpha}|{R}_{k0q_j\bar{q}_l}|\Big).
\end{eqnarray*}
\indent Now we go to estimate $\la F \ra_{r,D(s-\sigma,r)}^{+,\mathfrak{L}}$. We only estimate $\displaystyle\Big\|\frac{\partial^2 F}{\partial z_j\partial \bar{z}_l}\Big\|^{\mathfrak{L}}_{D(s-\sigma,r)}.$
Note $\frac{\partial^2 F}{\partial z_j\partial \bar{z}_l}=\sum_{k\in \Z^n}F_{k0q_{j}\bar{q}_l}e^{{\rm i}k\theta}$ and
$\Delta\frac{\partial^2 F}{\partial z_j\partial \bar{z}_l}=\sum_{k\in \Z^n}\Delta F_{k0q_{j}\bar{q}_l}e^{{\rm i}k\theta}$, it follows that
for $(\theta,y,z,\bar{z})\in D(s-\sigma,r),$
\begin{eqnarray*}
\frac{|\Delta\frac{\partial^2 F}{\partial z_j\partial \bar{z}_l}|}{|\xi-\eta|}&\leq &\sum_k\frac{|\Delta{F}_{k0q_j\bar{q}_l}|}{|\xi-\eta|}e^{|k|(s-\sigma)}\\
&\preceq &\sum_k\frac{ |k|A_k^2}{\alpha \la j-l\ra}\Big(\frac{|\Delta{R}_{k0q_j\bar{q}_l}|}{|\xi-\eta|}+\frac{M}{\alpha}|{R}_{k0q_j\bar{q}_l}|\Big)
e^{|k|(s-\sigma)}\\
&\preceq&\sum_k\frac{|k|A_k^2}{\alpha \la j-l\ra}\Big(|{R}_{k0q_j\bar{q}_l}|^{\mathfrak{L}}+\frac{M}{\alpha}|{R}_{k0q_j\bar{q}_l}|\Big)
e^{|k|(s-\sigma)}.
\end{eqnarray*}
Combining with $|{R}_{k0q_j\bar{q}_l}|\preceq\frac{\la R\ra_{r,D(r,s)}e^{-|k|s}}{(1+\ln j)^\beta(1+\ln l)^\beta}$  and $|{R}_{k0q_j\bar{q}_l}|^{\mathfrak{L}}\preceq\frac{\la R\ra^{\mathfrak{L}}_{r,D(r,s)}e^{-|k|s}}{(1+\ln j)^\beta(1+\ln l)^\beta}$, we deduce that
$$\Big\|\frac{\partial^2 F}{\partial z_j\partial \bar{z}_l}\Big\|^{\mathfrak{L}}_{D(s-\sigma,r)}\preceq\frac{c(n,\beta)e^{6(\frac{8}{\sigma})^{t_1}}}{\alpha(1+\ln j)^\beta(1+\ln l)^\beta(1+|j-l|)}\Big(\frac{M}{\alpha}\la R \ra_{r,D(s,r)}+\la R \ra_{r,D(s,r)}^{\mathfrak{L}}\Big).$$
\indent A similar computation for other terms provides that
\begin{eqnarray*}
\la F \ra_{r,D(s-\sigma,r)}^{+,\mathfrak{L}}&\preceq &\frac{c(n,\beta)e^{6(\frac{8}{\sigma})^{t_1}}}{\alpha}\Big(\frac{M}{\alpha}\la R \ra_{r,D(s,r)}+\la R \ra_{r,D(s,r)}^{\mathfrak{L}}\Big).
\end{eqnarray*}
The estimates for $\widehat{N}$ are similar and we omit the details.
\end{proof}
Now we turn to the estimates on Poisson bracket.
\begin{lemma}\label{Lrfsemi}
Let $R\in\Gamma_{r,D(s,r)}^{\beta}$ and  $F\in\Gamma_{r,D(s-\sigma,r)}^{\beta,+}$ be both of degree 2, i.e., $R,\ F$ are of the forms (\ref{2.9}) and (\ref{fpre}), respectively. Then, 
for any $0<2\sigma<s$,
\begin{eqnarray}\label{poi01}
\la \{R,F\} \ra_{r,D(s-2\sigma,\frac{r}{2})}&\preceq &\frac{1}{\sigma}\la R \ra_{r,D(s,r)}\la F \ra_{r,D(s-\sigma,r)}^{+},\\
\la \{R,F\} \ra_{r,D(s-2\sigma,\frac{r}{2})}^{\mathfrak{L}}&\preceq &\frac{1}{\sigma}\big(\la R \ra_{r,D(s,r)}\la F \ra_{r,D(s-\sigma,r)}^{+,\mathfrak{L}}+\la F \ra_{r,D(s-\sigma,r)}^+\la R \ra_{r,D(s,r)}^{\mathfrak{L}}\big).\label{poilip}
\end{eqnarray}
\end{lemma}
\begin{proof}
For simplicity we denote $\la R\ra : = \la R\ra_{r,D(s,r)},$  $\la R\ra^{\mathfrak{L}} : = \la R\ra_{r,D(s,r)}^{\mathfrak{L}} $,
$\la F\ra^{+} : = \la F\ra_{r,D(s-\sigma,r)}^{+}$  and
$\la F\ra^{+,\mathfrak{L}} : = \la F\ra_{r,D(s-\sigma,r)}^{+,\mathfrak{L}}$. Note that
$$\{R, F\}=\sum_{k=1}^n\Big(\frac{\partial R}{\partial\theta_k}\frac{\partial F}{\partial y_k}-\frac{\partial R}{\partial y_k}\frac{\partial F}{\partial\theta_k}\Big)+{\rm i}\sum_{j\geq1}\Big(\frac{\partial R}{\partial z_j}\frac{\partial F}{\partial \bar{z}_j}-\frac{\partial R}{\partial \bar{z}_j}\frac{\partial F}{\partial z_j}\Big),$$
it remains to estimate each term of this expansion and its derivatives.\\
\indent We first prove (\ref{poi01}).
From Cauchy and the basic inequality
$\displaystyle\sum\limits_{j\geq1}\frac{1}{j(1+\ln j)^{2\beta}}\preceq 1(\beta\geq 1)$,
we have
\begin{eqnarray}\label{poi02}
\|\{R, F\}\|_{D(s-2\sigma,r)}\preceq\frac{ r^2\la R\ra\la F\ra^+}{\sigma}.
\end{eqnarray}
\indent Similarly,
\begin{eqnarray}\label{poi03}
\max_{1\leq j\leq n}\Big\|\frac{\partial}{\partial y_j}\{R, F\}\Big\|_{D(s-2\sigma,r)}\preceq\frac{ \la R\ra\la F\ra^+}{\sigma}.
\end{eqnarray}
\indent Write
\begin{eqnarray}
\frac{\partial}{\partial z_j}\{R, F\}&=&\sum_{k=1}^n\Big(\frac{\partial^2 R}{\partial\theta_k\partial z_j}\frac{\partial F}{\partial y_k}+\frac{\partial R}{\partial\theta_k}\frac{\partial^2 F}{\partial y_k\partial z_j}-\frac{\partial^2 R}{\partial y_k\partial z_j}\frac{\partial F}{\partial\theta_k}-\frac{\partial R}{\partial y_k}\frac{\partial^2 F}{\partial\theta_k\partial z_j}\Big)\nonumber\\
&+&{\rm i}\sum_{k\geq1}\Big(\frac{\partial^2 R}{\partial z_k\partial z_j}\frac{\partial F}{\partial \bar{z}_k}+\frac{\partial R}{\partial z_k}\frac{\partial^2 F}{\partial \bar{z}_k\partial z_j}-\frac{\partial^2 R}{\partial \bar{z}_k\partial z_j}\frac{\partial F}{\partial z_k}-\frac{\partial R}{\partial \bar{z}_k}\frac{\partial^2 F}{\partial z_k\partial z_j}\Big)\nonumber\\
&:=& (I)+(II).\label{rfpianz1}
\end{eqnarray}
By a direct computation it holds that
\begin{eqnarray*}
\|(I)\|_{D(s-2\sigma,\frac{r}{2})}&\leq&\sum_{k=1}^n\Big(\frac{r\la R \ra\la F \ra^+}{\sigma(1+\ln j)^\beta}+\frac{4r\la R \ra\la F \ra^+}{\sigma j(1+\ln j)^\beta}+\frac{4r\la R \ra\la F \ra^+}{\sigma (1+\ln j)^\beta}+\frac{r\la R \ra\la F \ra^+}{\sigma j(1+\ln j)^\beta}\Big)\\
&\preceq&\frac{ r\la R \ra\la F \ra^+}{\sigma(1+\ln j)^\beta}.
\end{eqnarray*}
From Lemma \ref{basic1},
$
\|(II)\|_{D(s-2\sigma,\frac{r}{2})}\preceq\frac{ r\la R \ra\la F \ra^+}{(1+\ln j)^\beta}.
$
Thus
\begin{eqnarray}\label{poi04}
\Big\|\frac{\partial}{\partial z_j}\{R, F\}\Big\|_{D(s-2\sigma,\frac{r}{2})}&\preceq&\frac{ r\la R \ra\la F \ra^+}{\sigma(1+\ln j)^\beta}.
\end{eqnarray}
\indent By the same method and Lemma \ref{basic1} we obtain
\begin{eqnarray}\label{poi05}
\Big\|\frac{\partial^2}{\partial z_j\partial z_l}\{R, F\}\Big\|_{D(s-2\sigma,\frac{r}{2})}&\preceq&\frac{  \la R \ra\la F \ra^+}{\sigma(1+\ln j)^\beta(1+\ln l)^\beta}.
\end{eqnarray}
Together with (\ref{poi02})-(\ref{poi05}), (\ref{poi01}) is proved.\\
\indent For the  Lipschitz norm estimates, we first estimate $\|\frac{\partial}{\partial z_j}\{R, F\}\|^{\mathfrak{L}}_{D(s-2\sigma,\frac{r}{2})}$. Note that
\begin{eqnarray*}
\Big\|\frac{\partial^2 R}{\partial\theta_k\partial z_j}\frac{\partial F}{\partial y_k}\Big\|_{D(s-2\sigma,\frac{r}{2})}^{\mathfrak{L}}&\leq&\Big\|\frac{\partial^2 R}{\partial\theta_k\partial z_j}\Big\|_{D(s-2\sigma,\frac{r}{2})}^{\mathfrak{L}}\Big\|\frac{\partial F}{\partial y_k}\Big\|_{D(s-2\sigma,\frac{r}{2})}+\Big\|\frac{\partial^2 R}{\partial\theta_k\partial z_j}\Big\|_{D(s-2\sigma,\frac{r}{2})}\Big\|\frac{\partial F}{\partial y_k}\Big\|_{D(s-2\sigma,\frac{r}{2})}^{\mathfrak{L}}\\
\end{eqnarray*}
where
\begin{eqnarray*}
\Big\|\frac{\partial^2 R}{\partial\theta_k\partial z_j}\Big\|_{D(s-2\sigma,\frac{r}{2})}&\leq&\frac{1}{\sigma}\Big\|\frac{\partial R}{\partial z_j}\Big\|_{D(s,r)}\leq \frac{r\la R\ra}{\sigma(1+\ln j)^\beta},\\
\Big\|\frac{\partial^2 R}{\partial\theta_k\partial z_j}\Big\|_{D(s-2\sigma,\frac{r}{2})}^{\mathfrak{L}}&\leq&\frac{1}{\sigma}\Big\|\frac{\partial R}{\partial z_j}\Big\|_{D(s,r)}^{\mathfrak{L}}\leq \frac{r\la R\ra^{\mathfrak{L}}}{\sigma(1+\ln j)^\beta},
\end{eqnarray*}
$\Big\|\frac{\partial F}{\partial y_k}\Big\|_{D(s-2\sigma,\frac{r}{2})}\leq\la F\ra^+$ and
$\Big\|\frac{\partial F}{\partial y_k}\Big\|^{\mathfrak{L}}_{D(s-2\sigma,\frac{r}{2})}\leq\la F\ra^{+,\mathfrak{L}}$.
Hence
\begin{eqnarray*}
\Big\|\frac{\partial^2 R}{\partial\theta_k\partial z_j}\frac{\partial F}{\partial y_k}\Big\|_{D(s-2\sigma,\frac{r}{2})}^{\mathfrak{L}}&\preceq &\frac{r(\la R\ra^{\mathfrak{L}}\la F\ra^++\la R\ra\la F\ra^{+,\mathfrak{L}})}{\sigma(1+\ln j)^\beta}.\\
\end{eqnarray*}
For the other terms in $(I)$ in (\ref{rfpianz1}) the estimates are similar.
Thus,
\begin{eqnarray*}
\|(I)\|_{D(s-2\sigma,\frac{r}{2})}^{\mathfrak{L}}&\preceq &\frac{r(\la R\ra^{\mathfrak{L}}\la F\ra^++\la R\ra\la F\ra^{+,\mathfrak{L}})}{\sigma (1+\ln j)^\beta}.\\
\end{eqnarray*}
For (II) we only estimate
$\Big\|\sum\limits_{k\geq1}\frac{\partial^2 R}{\partial z_k\partial z_j}\frac{\partial F}{\partial \bar{z}_k}\Big\|_{D(s-2\sigma,\frac{r}{2})}^{\mathfrak{L}}$. From Lemma \ref{basic1},
\begin{eqnarray*}
&&\Big\|\sum_k\frac{\partial^2 R}{\partial z_k\partial z_j}\frac{\partial F}{\partial \bar{z}_k}\Big\|_{D(s-2\sigma,\frac{r}{2})}^{\mathfrak{L}}\\
&\leq&\sum_k\Big\|\frac{\partial^2 R}{\partial z_k\partial z_j}\Big\|_{D(s-2\sigma,\frac{r}{2})}^{\mathfrak{L}}\Big\|\frac{\partial F}{\partial \bar{z}_k}\Big\|_{D(s-2\sigma,\frac{r}{2})}+\sum_k\Big\|\frac{\partial^2 R}{\partial z_k\partial z_j}\Big\|_{D(s-2\sigma,\frac{r}{2})}\Big\|\frac{\partial F}{\partial \bar{z}_k}\Big\|_{D(s-2\sigma,\frac{r}{2})}^{\mathfrak{L}}\\
&\leq&\sum_k\frac{\la R\ra^{\mathfrak{L}}}{(1+\ln k)^\beta(1+\ln j)^\beta}\frac{r\la F\ra^+}{k(1+\ln k)^\beta}+\sum_k\frac{\la R\ra}{(1+\ln k)^\beta(1+\ln j)^\beta}\frac{r\la F\ra^{+,\mathfrak{L}}}{k(1+\ln k)^\beta}\\
&\preceq &\frac{r}{(1+\ln j)^\beta}(\la R\ra^{\mathfrak{L}}\la F\ra^++\la R\ra\la F\ra^{+,\mathfrak{L}}).
\end{eqnarray*}
Similar other estimates result in
\begin{eqnarray*}
\|(II)\|_{D(s-2\sigma,\frac{r}{2})}^{\mathfrak{L}}&\preceq &\frac{r}{(1+\ln j)^\beta}(\la R\ra^{\mathfrak{L}}\la F\ra^++\la R\ra\la F\ra^{+,\mathfrak{L}}).
\end{eqnarray*}
Therefore,
\begin{eqnarray*}
\Big\|\frac{\partial}{\partial z_j}\{R, F\}\Big\|_{D(s-2\sigma,\frac{r}{2})}^{\mathfrak{L}}&\preceq &\frac{r}{\sigma (1+\ln j)^\beta}\big(\la R\ra^{\mathfrak{L}}\la F\ra^++\la R\ra\la F\ra^{+,\mathfrak{L}}\big).
\end{eqnarray*}
To obtain (\ref{poilip}) we need some other estimates and the proofs are similar. We omit them for simplicity.
\end{proof}
\subsection{Phase flow.}
In this subsection we study the Hamiltonian flow generated by   $F\in\Gamma_{r,D(s-\sigma,r)}^{\beta,+}$  which is  globally of degree 2. Namely, we consider the system
\begin{eqnarray}\label{flow}
\left\{
\begin{array}{c}
(\dot{\theta}(t),\dot{y}(t),\dot{z}(t),\dot{\bar{z}}(t))=X_F({\theta}(t),{y}(t),{z}(t),{\bar{z}}(t)),\\
\\
({\theta}(0),{y}(0),{z}(0),{\bar{z}}(0))=({\theta}^0,{y}^0,{z}^0,{\bar{z}}^0).\ \ \ \ \ \ \ \ \ \  \ \ \ \
\end{array}
\right.
\end{eqnarray}
\begin{lemma}\label{flowmap}
Let $0<3\sigma<s$ and  $F\in\Gamma_{r,D(s-\sigma,r)}^{\beta,+}$ be  of degree 2. Assume that
\begin{eqnarray}\label{tiaojianforF}
\la F \ra_{r,D(s-\sigma,r)}^{+}<C\sigma.
\end{eqnarray}
Then the solution of Eq. (\ref{flow}) with the initial condition $({\theta}^0,{y}^0,{z}^0,{\bar{z}}^0)\in D(s-3\sigma,\frac{r}{4})$ satisfies $({\theta}(t),{y}(t),{z}(t),{\bar{z}}(t))\in D(s-2\sigma,\frac{r}{2})$ for all $0\leq t\leq 1,$ and we have the estimates
\begin{eqnarray}
\sup_{0\leq t\leq 1}\Big|\frac{\partial y_k(t)}{\partial w_j^0}\Big|&\preceq& \frac{r}{\sigma(1+\ln j)^\beta}\la F \ra_{r,D(s-\sigma,r)}^{+},\ \ \label{flowmap1}\\
\sup_{0\leq t\leq 1}\Big|\frac{\partial w_k(t)}{\partial w_j^0}\Big|
&\preceq &\frac{1}{(1+\ln j)^\beta(1+\ln k)^\beta(1+|j-k|)}\la F \ra_{r,D(s-\sigma,r)}^{+}+\delta_{jk},\label{flowmap2}\\
\sup_{0\leq t\leq 1}\Big|\frac{\partial y_k(t)}{\partial y_j^0}\Big|&\preceq& \frac{1}{\sigma}\la F \ra_{r,D(s-\sigma,r)}^{+}+\delta_{jk},\label{flowmap3}\\
\sup_{0\leq t\leq 1}\Big|\frac{\partial^2 y_k(t)}{\partial w_j^0\partial w_i^0}\Big|
&\preceq&\frac{1}{\sigma(1+\ln j)^\beta(1+\ln i)^\beta(1+|j-i|)}\la F \ra_{r,D(s-\sigma,r)}^{+}\label{flowmap4}
\end{eqnarray}
 with $ w_k =z_k$ or\ $\bar{z}_k  $ and $ w_k^0=z_k^0$\ or\ $\bar{z}_k^0,\ k=1,2,\cdots.$
\end{lemma}
Before we give the proof of Lemma \ref{flowmap} we  introduce a space of infinite dimensional matrices with decaying coefficients.
We denote by $\mathcal{M}$ the set of infinite matrices $A: \Z_{+}\times \Z_{+} \mapsto  \mathcal{M}_{2\times 2}(\C)$ with values in the space of complex $2\times 2$ matrices
and
$$|A| : = \sup\limits_{i,j \geq 1}\|A_{ij}\|_{HS}<\infty,$$
where $\|\cdot \|_{HS}$ denotes the Hilbert Schmidt norm:
$$\|M\|_{HS}^2 : = \sum\limits_{k,l=1}^2|M_{kl}|^2,$$
where $M\in \mathcal{M}_{2\times 2}(\C)$.
For $\beta>0$ we define $\mathcal{M}_{\beta}$ the subset of $\mathcal{M}$ such that
$[A]_{\beta} <\infty ,$
where the norm   $[\cdot]_{\beta}$
is given by the condition
\begin{eqnarray*}
\sup_{\xi\in\Pi}\|A_{ij}\|_{HS}&\leq &\frac{[ A]_{\beta}}{(1+\ln i)^\beta(1+\ln j)^\beta(1+|i-j|)},\ \ i,j\geq1.
 \end{eqnarray*}
In the following lemma we will suppress the parameter $\xi$  for simplicity.
\begin{lemma}\label{seminorm}
Let $A,B\in\mathcal{M}_{\beta}$  where $\xi\in \Pi$.  Then $A\cdot B\in\mathcal{M}_{\beta}$ and
$
[A \cdot B]_{\beta}\preceq[A]_{\beta}[ B]_{\beta}.
$
 \end{lemma}
\begin{proof}
For all $j,l\geq1$, $(A\cdot B)_{jl}=\sum_{k}A_{jk}B_{kl}.$ Thus, for $\xi\in \Pi$,
\begin{eqnarray*}
\|(A\cdot B)_{jl}\|_{HS}&\leq&\sum_{k\geq1}\|A_{jk}\|_{HS}\|B_{kl}\|_{HS}\\
&\leq&\sum_{k\geq1}\frac{[A]_{\beta}}{(1+\ln j)^\beta(1+\ln k)^\beta(1+|j-k|)}\frac{[B]_{\beta}}{(1+\ln k)^\beta(1+\ln l)^\beta(1+|k-l|)}\\
&\leq&\frac{[A]_{\beta}^+[ B]_{\beta}^+}{(1+\ln j)^\beta(1+\ln l)^\beta}\sum_{k\geq1}\frac{1}{(1+\ln k)^{2\beta}(1+|k-l|)(1+|j-k|)}\\
&\preceq&\frac{ [A]_{\beta}[B]_{\beta}}{(1+\ln j)^\beta(1+\ln l)^\beta(1+|j-l|)}.
\end{eqnarray*}
The last inequality comes from $\beta\geq 1$, Lemma \ref{basic1} and a similar discussion as Lemma 3.6 in \cite{GT}.
 \end{proof}
\noindent{Proof of Lemma \ref{flowmap}:}
For simplicity we introduce the notations $\zeta_j=(z_j,\bar{z}_j)$ and $\zeta=(\zeta_j)_{j\geq1}.$ Then $F$ reads
\begin{eqnarray}\label{formforF}
F(\theta,y,\zeta)=a_0(\theta)+a_1(\theta)\cdot y+b(\theta)\cdot \zeta+\frac12(B(\theta)\zeta)\cdot \zeta
\end{eqnarray}
with $a_0(\theta)=F(\theta,0,0),\ a_1(\theta)=\nabla_y F(\theta,0,0),\ b(\theta)=\nabla_{\zeta} F(\theta,0,0)$ and $B=(B_{ij})$ is the infinite matrix where
\begin{eqnarray}\label{bt}
B_{ij}(\theta)=\left(
\begin{array}{cc}
\frac{\partial^2 F}{\partial z_i\partial z_j}(\theta,0,0) & \frac{\partial^2 F}{\partial z_i\partial \bar{z}_j}(\theta,0,0)\\
\frac{\partial^2 F}{\partial \bar{z}_i\partial z_j}(\theta,0,0) & \frac{\partial^2 F}{\partial \bar{z}_i\partial \bar{z}_j}(\theta,0,0)
\end{array}
\right).
\end{eqnarray}
The flow $X^t_F$
exists for $0 \leq t \leq 1$ and maps $D(s-3\sigma, r/4)$ into
$D(s-2\sigma, r/2)$. In the sequel we write
\begin{eqnarray}\label{thetat}({\theta}(t),{y}(t),{\zeta}(t))=X_F^t({\theta}^0,{y}^0,{\zeta}^0).
\end{eqnarray}
From the equation   $\dot{\theta}=\nabla_y F(\theta,y,\zeta)$
and (\ref{tiaojianforF}),  we have the bound
$\sup_{0\leq t\leq1}|\Im \theta(t)|<s-2\sigma$. \\
\indent We now turn to the equation in $\zeta$. To solve
$$ \dot \zeta(t) =J\nabla_{\zeta} F(\theta, y, \zeta)(t)=b_1(t)+\mathcal{B}(t)\zeta(t), \ \ \zeta(0) = \zeta_0, $$
where $b_1(t)=Jb(\theta(t))$ and $\mathcal{B}(t)=JB(\theta(t))$ with
\begin{eqnarray*}
J=diag\left\{\left(
\begin{array}{cc}
0 & 1\\
-1 & 0
\end{array}
\right)\right\}_{j\geq1},
\end{eqnarray*}
we iterate the integral formulation of the problem
$$\zeta(t) =\zeta^0+\int_0^t(b_1(t_1)+\mathcal{B}(t_1)\zeta(t_1))dt_1, $$
and formally obtain
\begin{eqnarray}\label{zt}
\zeta(t) =b^\infty(t)+(1+\mathcal{B}^\infty(t))\zeta^0,
\end{eqnarray}
where
\begin{eqnarray}\label{binfty}b^\infty(t)=\sum_{k\geq1}\int_0^t\int_0^{t_1}\cdots\int_0^{t_{k-1}}\prod_{1\leq j\leq k-1}\mathcal{B}(t_j)b_1(t_k)dt_k\cdots dt_1,\end{eqnarray}
and
\begin{eqnarray}\label{bbinfty}\mathcal{B}^\infty(t)=\sum_{k\geq1}\int_0^t\int_0^{t_1}\cdots\int_0^{t_{k-1}}\prod_{1\leq j\leq k}\mathcal{B}(t_j)dt_k\cdots dt_1.\end{eqnarray}
It is clear that there exists $C>0$ so that
$$\sup_{0\leq t\leq1}\|\mathcal{B}(t)\|_{\mathcal{L}(\ell^{2,p}, \ell^{2,p})}\leq C,$$
and thus, for all $0 \leq t\leq1$ the series (\ref{binfty}) converges by
\begin{eqnarray*}
\|b^\infty(t)\|_{\ell^{2,p}} \leq \sup_{0\leq t\leq1}\|b_1(t)\|_{\ell^{2,p}}\sum_{k\geq1}\frac{(4\la F\ra^+)^{k-1}}{k!}\leq e\sup_{0\leq t\leq1}\|b_1(t)\|_{\ell^{2,p}}.
\end{eqnarray*}
Similarly for $0\leq t\leq1$,
$\|\mathcal{B}^\infty(t)\|_{\mathcal{L}(\ell^{2,p}, \ell^{2,p})}\leq 4e\la F\ra^+$.
As a conclusion, the formula (\ref{zt}) makes sense. In fact we can say more about $\mathcal{B}^{\infty}(t)$. For $|\Im \theta|<s-2\sigma$,
$$|\mathcal{B}_{ij}^{11}|=|B_{ij}^{21}|=\Big|\frac{\partial^2 F}{\partial \bar{z}_i\partial z_j}(\theta,0,0) \Big|\leq \frac{\la F\ra^+}{(1+\ln i)^\beta(1+\ln j)^\beta(1+|i-j|)}.
 $$ Similar estimates hold for $\mathcal{B}_{ij}^{12}$, $\mathcal{B}_{ij}^{21}$ and $\mathcal{B}_{ij}^{22}$. Recall that $\mathcal{B}(t) =JB(\theta(t))$ and
 $|\Im \theta(t)|<s-2\sigma$ for $0\leq t\leq 1$.
 It follows $\mathcal{B}(t)\in \mathcal{M}_{\beta}$ and
$\sup\limits_{0\leq t\leq 1}[ \mathcal{B}(t)]_{\beta}\leq \la F\ra^+$.
Hence, by Lemma \ref{seminorm} and (\ref{bbinfty}),
\begin{eqnarray}\label{gujiforbinfty}
\sup\limits_{0\leq t\leq 1}[\mathcal{B}^\infty(t)]_{\beta}\leq e^{\la F\ra^+}-1\leq e\la F\ra^+_{r,D(s-\sigma,r)}.
\end{eqnarray}
\indent In the following we study the equation in $y$,
$ \dot y(t) =-\nabla_\theta F(\theta, y, \zeta)(t), \ \ y(0) = y_0. $ From (\ref{formforF}),
$$ \dot y(t) = f(t)+g(t) y(t), \ \ y(0) = y_0, $$
where $f(t)=-\nabla_{\theta} a_0(\theta(t))-\nabla_{\theta}b(\theta(t))\zeta-\frac12 (\nabla_{\theta}B(\theta(t))\zeta)\cdot \zeta $
and $g(t) =-\nabla_{\theta}\nabla_{y} F(\theta,0,0)$.
As above, we have formally
\begin{eqnarray}\label{yt}
y(t) =f^\infty(t)+(1+g^\infty(t))y^0,
\end{eqnarray}
where
\begin{eqnarray*}\label{finfty}f^\infty(t)=\sum_{k\geq1}\int_0^t\int_0^{t_1}\cdots\int_0^{t_{k-1}}\prod_{1\leq j\leq k-1}g(t_j)f(t_k)dt_k\cdots dt_1,\end{eqnarray*}
and
\begin{eqnarray*}\label{ginfty}g^\infty(t)=\sum_{k\geq1}\int_0^t\int_0^{t_1}\cdots\int_0^{t_{k-1}}\prod_{1\leq j\leq k}g(t_j)dt_k\cdots dt_1.\end{eqnarray*}
From Cauchy we have
$$\sup_{0\leq t\leq1}\|g(t)\|\preceq \frac{1}{\sigma}\max_{1\leq j\leq n}\Big|\frac{\partial F}{\partial y_j}(\theta(t),0,0)\Big|\leq \frac{1}{\sigma}\la F\ra^+_{r,D(s-\sigma,r)}:= \kappa,$$
which follows that for $0 \leq t\leq1$,
\begin{eqnarray*}
\|f^\infty(t)\|&\leq&\sum_{k\geq1}\int_0^t\int_0^{t_1}\cdots\int_0^{t_{k-1}}\prod_{1\leq j\leq k-1}\kappa^{k-1}\|f(t)\|dt_k\cdots dt_1\\
&\leq&\sup_{0\leq t\leq1}\|f(t)\|\sum_{k\geq1}\frac{\kappa^{k-1}}{k!}
\preceq \sup_{0\leq t\leq1}\|f(t)\|.
\end{eqnarray*}
Similarly for $0\leq t\leq1,$
\begin{eqnarray}\label{ginfty}
\|g^\infty(t)\| \preceq \frac{ \la F\ra^+}{\sigma}.
\end{eqnarray}
Therefore (\ref{yt}) makes sense.\\
\indent  Now we turn to show the estimates on the solutions of the equations (\ref{flow}).
By (\ref{zt}),
\begin{eqnarray}\label{ztdi}
\nabla_{\zeta_j^0}\zeta_k(t) =\left(\begin{array}{cc}
1 & 0\\
0 & 1
\end{array}\right)
\delta_{kj}+\mathcal{B}_{kj}^\infty(t),
\end{eqnarray}
and (\ref{gujiforbinfty}) we have (\ref{flowmap2}).
From
$
y_k(t) =f_k^\infty(t)+y_k^0+\sum_{1\leq j\leq n}g_{jk}^\infty(t)y_j^0
$
and (\ref{ginfty}) we obtain (\ref{flowmap3}).
In the following we give the estimates (\ref{flowmap1}) and (\ref{flowmap4}). \\
\indent Since $g$ and $g^\infty$ do not depend on $\zeta$, we obtain that $\displaystyle\frac{\partial y(t)}{\partial z_j^0}=\frac{\partial f^\infty}{\partial z_j^0}$. Now by the definition of $f^\infty$, we deduce  that, for $0\leq t\leq1,$
\begin{eqnarray*}\label{}
\Big\|\frac{\partial y(t)}{\partial z_j^0}\Big\|&=&\Big\|\sum_{k\geq1}\int_0^t\int_0^{t_1}\cdots\int_0^{t_{k-1}}\prod_{1\leq j\leq k-1}g(t_j)\frac{\partial f(t_k)}{\partial z_j^0}dt_k\cdots dt_1\Big\|\\
&\leq& \sum_{k\geq1}\int_0^1\int_0^{t_1}\cdots\int_0^{t_{k-1}}\prod_{1\leq j\leq k-1}\|g(t_j)\|\Big\|\frac{\partial f(t_k)}{\partial z_j^0}\Big\|dt_k\cdots dt_1\\
&\leq& \sum_{k\geq1}\int_0^t\int_0^{t_1}\cdots\int_0^{t_{k-1}}B^{k-1}\Big\|\frac{\partial f(t_k)}{\partial z_j^0}\Big\|dt_k\cdots dt_1\\
&\preceq&  \sup_{0\leq t \leq1}|\nabla_{\zeta_j^0}f(t)|.
\end{eqnarray*}
From a straightforward computation we have, for all $1\leq l\leq n$,
\begin{eqnarray}\label{zk}
\nabla_{\zeta_k}f_l(t)=-\partial_{\theta_l}b_k(\theta(t))-\sum_{i\geq 1}\partial_{\theta_l}B_{ki}(\theta(t))\zeta_i(t),\ \ {\rm with}\ b_k(\theta)=\nabla_{\zeta_k}F(\theta,0,0).
\end{eqnarray}
By Cauchy we obtain that
\begin{eqnarray*}\label{}
\sup_{0\leq t\leq1}|\partial_{\theta_l}b_k(\theta(t))|
\preceq\frac{1}{\sigma}\sup\limits_{|\Im \theta|<s-2\sigma}|\nabla_{\zeta_k}F(\theta,0,0)|\preceq\frac{1}{\sigma}\frac{r\la F \ra^+_{r,D(s-\sigma,r)}}{k(1+\ln k)^\beta}.
\end{eqnarray*}
For the second term in (\ref{zk}) we obtain by Cauchy
\begin{eqnarray*}\label{}
\sup_{0\leq t\leq1}\|\partial_{\theta_l}B_{ki}(\theta(t))\|&\preceq &\frac{\la F \ra^+_{r,D(s-\sigma,r)}}{\sigma(1+|k-i|)(1+\ln k)^\beta(1+\ln i)^\beta}.
\end{eqnarray*}
 Thus,
\begin{eqnarray*}\label{}
\|\nabla_{\zeta_k}f_l(t)\|&\preceq& \frac{r\la F \ra^+_{r,D(s-\sigma,r)}}{\sigma k(1+\ln k)^\beta}+\sum_{i\geq1}\frac{\la F \ra^+_{r,D(s-\sigma,r)}}{\sigma(1+|k-i|)(1+\ln k)^\beta(1+\ln i)^\beta}\|\zeta_i\|\\
&\preceq& \frac{\la F \ra^+_{r,D(s-\sigma,r)}}{\sigma(1+\ln k)^\beta}\Big({r}+\sum_{i\geq1}\frac{\|\zeta_i\|}{(1+|k-i|)(1+\ln i)^\beta}\Big)\\
&\preceq& \frac{r\la F \ra^+_{r,D(s-\sigma,r)}}{\sigma(1+\ln k)^\beta}.
\end{eqnarray*}
Further, from
$
\nabla_{\zeta_j^0}f_l(t)=\sum\limits_{k\geq 1}\nabla_{\zeta_j^0}\zeta_k\nabla_{\zeta_k}f_l(t),
$
we have
\begin{eqnarray*}\label{}
\|\nabla_{\zeta_j^0}f_l(t)\|&\leq&\sum_{k\geq 1}\|\nabla_{\zeta_j^0}\zeta_k\|\|\nabla_{\zeta_k}f_l(t)\|\\
&\preceq&\sum_{k\geq 1,k\neq j}\frac{ \la F \ra^+_{r,D(s-\sigma,r)}}{(1+|k-j|)(1+\ln k)^\beta(1+\ln j)^\beta}
\cdot \frac{r\la F \ra^+_{r,D(s-\sigma,r)}}{\sigma(1+\ln k)^\beta}+ \frac{r\la F \ra^+_{r,D(s-\sigma,r)}}{\sigma(1+\ln j)^\beta}\\
&\preceq& \frac{r\la F \ra^+_{r,D(s-\sigma,r)}}{\sigma(1+\ln j)^\beta}\Big(1+\sum_{k\geq 1,k\neq j}\frac{1}{(1+|k-j|)(1+\ln k)^{2\beta}}\Big)\\
&\preceq&\frac{ r\la F \ra^+_{r,D(s-\sigma,r)}}{\sigma(1+\ln j)^\beta}.
\end{eqnarray*}
The above third inequality comes from Lemma \ref{basic1} and $\beta\geq 1$.
It then follows
\begin{eqnarray*}\label{}
\sup_{0\leq t\leq 1}\Big\|\frac{\partial y_k(t)}{\partial z_j^0}\Big\|\preceq\frac{ r\la F \ra^+_{r,D(s-\sigma,r)}}{\sigma(1+\ln j)^\beta}.
\end{eqnarray*}
\indent It remains to show (\ref{flowmap4}).
First we have
\begin{eqnarray*}\label{}
\sup_{0\leq t\leq 1}\Big|\frac{\partial^2 y_k(t)}{\partial z_j^0\partial z_i^0}\Big| \preceq\sup_{0\leq t\leq1}\|\nabla_{\zeta_i^0}\nabla_{\zeta_j^0}f(t)\|.
\end{eqnarray*}
Note $\|\nabla_{\zeta_i^0}\nabla_{\zeta_j^0}f(t)\|=\|\nabla_\theta B_{ij}(\theta(t))\|$ and use Cauchy in $\theta$,
\begin{eqnarray*}\label{}
\Big|\frac{\partial^2 y_k(t)}{\partial z_j^0\partial z_i^0}\Big| \preceq\frac{ \la F \ra^+_{r,D(s-\sigma,r)}}{\sigma(1+\ln i)^\beta(1+\ln j)^\beta(1+|i-j|)}.
\end{eqnarray*}
\qed\\
\indent Similarly, we have
\begin{lemma}\label{daoshul}
Under the assumptions of Lemma \ref{flowmap} and the condition $\la F \ra_{r,D(s-\sigma,r)}^{+,\mathfrak{L}}\leq C\sigma$,  the solution of (\ref{flow}) satisfies
\begin{eqnarray*}
\sup_{0\leq t\leq 1}\Big|\frac{\partial y_k(t)}{\partial w_j^0}\Big|^{\mathfrak{L}}&\preceq &\frac{ r}{\sigma(1+\ln j)^\beta}\la F \ra_{r,D(s-\sigma,r)}^{+,\mathfrak{L}},\\
\sup_{0\leq t\leq 1}\Big|\frac{\partial w_k(t)}{\partial w_j^0}\Big|^{\mathfrak{L}}&\preceq&\frac{1}{(1+\ln j)^\beta(1+\ln k)^\beta(1+|j-k|)}\la F \ra_{r,D(s-\sigma,r)}^{+,\mathfrak{L}},\ \\
\sup_{0\leq t\leq 1}\Big|\frac{\partial y_k(t)}{\partial y_j^0}\Big|^{\mathfrak{L}}&\preceq &\frac{1}{\sigma}\la F \ra_{r,D(s-\sigma,r)}^{+,\mathfrak{L}},\\
\sup_{0\leq t\leq 1}\Big|\frac{\partial^2 y_k(t)}{\partial w_j^0\partial w_i^0}\Big|^{\mathfrak{L}}&\preceq &\frac{1}{\sigma(1+\ln j)^\beta(1+\ln i)^\beta(1+|j-i|)}\la F \ra_{r,D(s-\sigma,r)}^{+,\mathfrak{L}}
\end{eqnarray*}
 with $ w_k =z_k$ or\ $\bar{z}_k  $ and $ w_k^0=z_k^0$\ or\ $\bar{z}_k^0,\ k=1,2,\cdots.$
\end{lemma}
The proof of Lemma \ref{flowmap} implies
\begin{Corollary}
The time 1 map $X_F^1$ reads
\begin{eqnarray*}
\left(
\begin{array}{c}
\theta\\
y\\
\zeta
\end{array}
\right)\mapsto\left(
\begin{array}{c}
K(\theta)\\
L(\theta,\zeta)+M(\theta)\zeta+S(\theta)y\\
T(\theta)+U(\theta)\zeta
\end{array}
\right),\end{eqnarray*}
where $L(\theta,\zeta)$ is quadratic in $\zeta$, $M(\theta)$ and $U(\theta)$ are bounded linear operators from
$\ell^{2,p}\times\ell^{2,p}$ into $\R^n$ and $\ell^{2,p}\times\ell^{2,p}$ respectively,  and $S(\theta)$ is a bounded linear map from $\R^n$ to $\R^n$.
\end{Corollary}

\subsection{Composition estimates}
\begin{Proposition}\label{map}
Let $0<\eta<1/8$  and $0 < \sigma < s$, $R\in\Gamma^\beta_{\eta r,D(s-2\sigma,4\eta r)}$  and $F \in \Gamma^{\beta,+}_{ r,D(s-\sigma,r)}$ with F of degree 2. Assume that
\begin{eqnarray}\label{gapflowmap}
\la F \ra_{r,D(s-\sigma,r)}^{+}+\la F \ra_{r,D(s-\sigma,r)}^{+,\mathfrak{L}}<C\sigma\eta^2.
\end{eqnarray}
 Then $R\circ X^1_F\in\Gamma^\beta_{\eta r,D(s-5\sigma,\eta r)}$
 and
 \begin{eqnarray}\label{rfsemi}
\la R\circ X^1_F \ra_{\eta r,D(s-5\sigma,\eta r)}&\preceq & \la R \ra_{\eta r,D(s-2\sigma,4\eta r)},
\end{eqnarray}
 \begin{eqnarray}\label{rflsemi}
\la R\circ X^1_F \ra_{\eta r,D(s-5\sigma,\eta r)}^{\mathfrak{L}}&\preceq &  \la R \ra_{\eta r,D(s-2\sigma,4\eta r)}+\la R \ra_{\eta r,D(s-2\sigma,4\eta r)}^{\mathfrak{L}}.
\end{eqnarray}
\end{Proposition}
\begin{remark}
In Proposition \ref{map},  we don't require  $R$ of degree 2.
\end{remark}
\begin{proof}
In the sequel, we use the notation
$$
({\theta},{y},z,\bar{z})=X_F^1({\theta}^0,{y}^0,{z}^0,\bar{z}^0).
$$
From (\ref{gapflowmap}) and $\la \cdot \ra_{4\eta r, D(s-\sigma, 4\eta r) }\preceq\eta^{-2}\la \cdot \ra_{r,D(s-\sigma,r)}$, it follows
\begin{eqnarray}\label{fetar}
\la F\ra_{4\eta r,D(s-\sigma,4\eta r)}\preceq\sigma
\end{eqnarray} which will be used later.  Now by (\ref{gapflowmap}) it is easy to show that $X_F^1$ maps $D(s-5\sigma,\eta r)$ into $D(s-2\sigma,4\eta r)$ and thus,
\begin{eqnarray}\label{rmap}
\| R\circ X^1_F \|_{D(s-5\sigma,\eta r)}&\leq &(\eta r)^{2}\la R \ra_{\eta r,D(s-2\sigma,4\eta r)}.
\end{eqnarray}
By the Leibniz rule, for all $1\leq j\leq n$,
\begin{eqnarray*}\label{}
 \frac{\partial(R\circ X^1_F)}{\partial y_j^0} =\sum_{k=1}^n\frac{\partial R(X^1_F)}{\partial y_k}\frac{\partial y_k}{\partial y_j^0}.
\end{eqnarray*}
From the definition
$
\Big\|\frac{\partial R}{\partial y_k}\Big\|_{D(s-2\sigma, 4\eta r)}\leq \la R \ra_{\eta r,D(s-2\sigma,4\eta r)},
$
and by Lemma \ref{flowmap},
\begin{eqnarray*}\label{}
\sup_{0\leq t\leq 1}\Big|\frac{\partial y_k(t)}{\partial y_j^0}\Big| \preceq  \frac{1}{\sigma}\la F \ra_{4\eta r,D(s-\sigma,4\eta r)}^{+}+\delta_{jk}.
\end{eqnarray*}
Thus, from (\ref{fetar}) we obtain
\begin{eqnarray}\label{ry}
 \Big|\frac{\partial(R\circ X^1_F)}{\partial y_j^0}\Big|\preceq  \la R \ra_{\eta r,D(s-2\sigma,4\eta r)}.
\end{eqnarray}
\indent For $j\geq1$, the derivatives in $z_j^0$ reads
\begin{eqnarray*}\label{}
 \frac{\partial(R\circ X^1_F)}{\partial z_j^0} =\sum_{k=1}^n\frac{\partial R(X^1_F)}{\partial y_k}\frac{\partial y_k}{\partial z_j^0}+\sum_{k\geq1}\Big(\frac{\partial R(X^1_F)}{\partial z_k}\frac{\partial z_k}{\partial z_j^0}+\frac{\partial R(X^1_F)}{\partial \bar{z}_k}\frac{\partial \bar{z}_k}{\partial z_j^0}\Big):= (I)+(II).
\end{eqnarray*}
From (\ref{fetar}),
\begin{eqnarray*}
 |(I)|&\preceq&\sum_{k=1}^n\frac{ \eta r}{\sigma(1+\ln j)^\beta}\la R \ra_{\eta r,D(s-2\sigma,4\eta r)}\la F \ra_{4\eta r,D(s-\sigma,4\eta r)}^{+} \nonumber\\
 &\preceq & \frac{ \eta r}{(1+\ln j)^\beta}\la R \ra_{\eta r,D(s-2\sigma,4\eta r)},
\end{eqnarray*}
and
\begin{eqnarray*}\label{}
 |(II)|&\leq&\sum_{k\geq1}\Big(\Big|\frac{\partial R(X^1_F)}{\partial z_k}\Big|\Big|\frac{\partial z_k}{\partial z_j^0}\Big|+\Big|\frac{\partial R(X^1_F)}{\partial \bar{z}_k}\Big|\Big|\frac{\partial \bar{z}_k}{\partial z_j^0}\Big|\Big)\\
 &\preceq& \sum_{k\geq1}\Big(\frac{ \la F \ra_{4\eta r,D(s-\sigma,4\eta r)}^{+}}{(1+\ln j)^\beta(1+\ln k)^\beta(1+|j-k|)}+\delta_{jk}
\Big)\frac{ \eta r}{ (1+\ln k)^\beta}\la R \ra_{\eta r,D(s-2\sigma,4\eta r)}\\
 &\preceq & \frac{ \eta r}{ (1+\ln j)^\beta}\la R \ra_{\eta r,D(s-2\sigma,4\eta r)}\Big(1+\la F \ra_{4\eta r,D(s-\sigma,4\eta r)}^{+}\sum_{k\geq 1}\frac{1}{(1+\ln k)^{2\beta}(1+|j-k|)}\Big)\\
\underline{(\ref{fetar})}&\preceq & \frac{ \eta r}{ (1+\ln j)^\beta}\la R \ra_{\eta r,D(s-2\sigma,4\eta r)}.
\end{eqnarray*}
Thus
\begin{eqnarray}\label{rz}
 \Big|\frac{\partial(R\circ X^1_F)}{\partial z_j^0}\Big|&\preceq&\frac{ \eta r}{ (1+\ln j)^\beta}\la R \ra_{\eta r,D(s-2\sigma,4\eta r)}.
\end{eqnarray}
\indent We now estimate
$\displaystyle
 \Big\|\frac{\partial^2(R\circ X^1_F)}{\partial z_i^0\partial z_j^0}\Big\|_{D(s-5\sigma,\eta r)}.
$
 The derivatives reads
\begin{eqnarray*}\label{}
 \frac{\partial^2(R\circ X^1_F)}{\partial z_i^0\partial z_j^0}& = & (I_1) + (I_2) + (I_3) + (I_4)
 \end{eqnarray*}
with
\begin{eqnarray*}
(I_1) = \sum_{k,l=1}^n\frac{\partial^2 R(X^1_F)}{\partial y_k\partial y_l}\frac{\partial y_l}{\partial z_i^0}\frac{\partial y_k}{\partial z_j^0},\quad\quad
(I_2) = \sum_{k=1}^n\frac{\partial R(X^1_F)}{\partial y_k}\frac{\partial^2 y_k}{\partial z_i^0\partial z_j^0},
\end{eqnarray*}
\begin{eqnarray*}\label{}
(I_3)&=&\sum_{k\geq1}\sum_{l=1}^n\frac{\partial^2 R(X^1_F)}{\partial y_l\partial z_k}\frac{\partial y_l}{\partial z_i^0}\frac{\partial {z}_k}{\partial z_j^0}+\sum\limits_{k\geq 1}\sum_{p\geq1}\frac{\partial^2 R(X^1_F)}{\partial z_p\partial z_k}\frac{\partial {z}_p}{\partial z_i^0} \frac{\partial {z}_k}{\partial z_j^0}+\sum\limits_{k\geq 1}\sum_{p\geq1}\frac{\partial^2 R(X^1_F)}{\partial \bar{z}_p\partial z_k}\frac{\partial {\bar{z}}_p}{\partial z_i^0}\frac{\partial {z}_k}{\partial z_j^0},\\
&=& (I)_a+(I)_b+(I)_c.
 \end{eqnarray*}
 and
 \begin{eqnarray*}\label{}
(I_4)=\sum_{k\geq1}\Big(\sum_{l=1}^n\frac{\partial^2 R(X^1_F)}{\partial y_l\partial \bar{z}_k}\frac{\partial y_l}{\partial z_i^0}+\sum_{p\geq1}\frac{\partial^2 R(X^1_F)}{\partial z_p\partial \bar{z}_k}\frac{\partial {z}_p}{\partial z_i^0}+\sum_{p\geq1}\frac{\partial^2 R(X^1_F)}{\partial \bar{z}_p\partial \bar{z}_k}\frac{\partial {\bar{z}}_p}{\partial z_i^0}\Big)\frac{\partial {\bar{z}}_k}{\partial z_j^0}.
 \end{eqnarray*}
 We give a detailed estimation for $(I_3)$.
From Cauchy, (\ref{fetar}) and Lemma \ref{flowmap},
\begin{eqnarray*}
\|(I)_a\|_{D(s-5\sigma,\eta r)}&\leq& \sum\limits_{k\geq 1}\sum\limits_{l=1}^n\Big|\frac{\partial^2 R\circ X_{F}^1}{\partial y_l\partial z_k}\Big|\cdot
\Big|\frac{\partial y_l}{\partial z_i^0}\Big| \cdot \Big|\frac{\partial z_k}{\partial z_j^0}\Big|\\
&\preceq  & \sum\limits_{k\geq 1}\sum\limits_{l=1}^n(\eta r)^{-2}\Big|\frac{\partial R}{\partial z_k}\Big|_{D(s-2\sigma,4\eta r)}\cdot
\frac{\eta r \la F\ra^+_{4\eta r,D(s-\sigma, 4\eta r)}}{\sigma(1+\ln i)^{\beta}}\cdot \Big(\frac{ \la F\ra^+_{4\eta r, D(s-\sigma, 4\eta r)}}{(1+\ln j)^{\beta}(1+\ln k)^{\beta}(1+|j-k|)}+\delta_{jk}\Big)\\
&\preceq  &\frac{ \la R\ra_{\eta r,D(s-2\sigma,4\eta r)}}{ (1+\ln i)^\beta (1+\ln j)^\beta}.
\end{eqnarray*}
By the same way,
\begin{eqnarray*}
\|(I)_b\|_{D(s-5\sigma,\eta r)}&\preceq & \sum\limits_{k,p\geq 1}\Big|\frac{\partial^2 R\circ X_{F}^1}{\partial z_p\partial z_k}\Big|\cdot
\Big|\frac{\partial z_p}{\partial z_i^0}\Big| \cdot \Big|\frac{\partial z_k}{\partial z_j^0}\Big|\\
&\preceq  & \sum\limits_{k,p\geq 1}\frac{\la R\ra_{\eta r, D(s-2\sigma,4\eta r)}}{(1+\ln k)^{\beta}(1+\ln p)^{\beta}} \cdot \Big(\frac{ \la F\ra^+_{4\eta r, D(s-\sigma, 4\eta r)}}{(1+\ln p)^{\beta}(1+\ln i)^{\beta}(1+|p-i|)}+\delta_{pi}\Big)\\
&&\cdot  \Big(\frac{ \la F\ra^+_{4\eta r, D(s-\sigma, 4\eta r)}}{(1+\ln j)^{\beta}(1+\ln k)^{\beta}(1+|j-k|)}+\delta_{jk}\Big)\\
&\preceq &\frac{ \la R\ra_{\eta r,D(s-2\sigma,4\eta r)}}{ (1+\ln i)^\beta (1+\ln j)^\beta}.
\end{eqnarray*}
Similarly, we have
$$\|(I)_c\|_{D(s-5\sigma,\eta r)}\preceq \frac{ \la R\ra_{\eta r,D(s-2\sigma,4\eta r)}}{ (1+\ln i)^\beta (1+\ln j)^\beta}.$$
Therefore,
$$\|(I_3)\|_{D(s-5\sigma,\eta r)}\preceq \frac{ \la R\ra_{\eta r,D(s-2\sigma,4\eta r)}}{ (1+\ln i)^\beta (1+\ln j)^\beta}.$$
The similar computation provides us
\begin{eqnarray*}
\|(I_1)\|_{D(s-5\sigma,\eta r)}&\preceq& \frac{ \la R\ra_{\eta r,D(s-2\sigma,4\eta r)}}{ (1+\ln i)^\beta (1+\ln j)^\beta},\\
\|(I_2)\|_{D(s-5\sigma,\eta r)}&\preceq& \frac{ \la R\ra_{\eta r,D(s-2\sigma,4\eta r)}}{ (1+\ln i)^\beta (1+\ln j)^\beta},\\
\|(I_4)\|_{D(s-5\sigma,\eta r)}&\preceq& \frac{ \la R\ra_{\eta r,D(s-2\sigma,4\eta r)}}{ (1+\ln i)^\beta (1+\ln j)^\beta}.
\end{eqnarray*}
It results in
\begin{eqnarray}\label{rzz}
 \Big\|\frac{\partial^2(R\circ X^1_F)}{\partial z_i^0\partial z_j^0}\Big\|_{D(s-5\sigma,\eta r)}\preceq \frac{ \la R\ra_{\eta r,D(s-2\sigma,4\eta r)}}{(1+\ln i)^\beta (1+\ln j)^\beta}.
\end{eqnarray}
\indent By (\ref{rmap}), (\ref{ry}), (\ref{rz}) and (\ref{rzz}), (\ref{rfsemi}) holds. We omit the proof of (\ref{rflsemi}), which is similar by using the estimates of Lemma \ref{daoshul} instead.
\end{proof}

\begin{lemma}\label{xpr}
Assume $P$ satisfies Assumption $\mathcal{B}$ and consider its Taylor approximation
$R$ of the form (\ref{2.9}). Then, for all $\eta > 0$,
$$\|X_R\|^*_{r,D(s,r )} \preceq  \|X_P\|^*_{r,D(s,r )} ,$$ and
$$\|X_P-X_R\|^*_{\eta r,D(s,4\eta r )} \preceq\eta\|X_P\|^*_{r,D(s,r )}.$$
\end{lemma}
We have an analogous result for the norm $\la\cdot\ra_{r,D(s,r )}$.
\begin{lemma}\label{p-rsemi}
Let $P\in\Gamma^{\beta}_{ r,D(s,r)}$ and consider its Taylor approximation $R$ of the form
(\ref{2.9}). Then, for all $\eta > 0$,
$$\la R\ra_{\eta r,D(s, r)}^*\preceq\la P\ra_{ r,D(s, r)}^*,$$
and
$$\la P-R\ra^*_{\eta r,D(s,4\eta r)}\preceq\eta\la P\ra_{ r,D(s, r)}^*.$$
\end{lemma}
We omit the proofs for the above two lemmas.
\subsection{The KAM Step}
Let $N$ be a Hamiltonian in normal form as in (\ref{1.2}), which reads in the variables
$(\theta, y, z, \bar{z})$,
$$N = \sum_{
1\leq j\leq n}
\omega_j (\xi) +\sum_{j\geq 1}
\Omega_j(\xi)z_j \bar{z}_j ,$$
and suppose that {  Assumption $\mathcal{A}$} is satisfied.
Consider a perturbation $P$ which satisfies {  Assumption $\mathcal{B}$} for some $r, s > 0$.
Then choose $0 <\eta < 1/8,\ 0 < \sigma < s$  and assume that
\begin{eqnarray}\label{iterationtiaojian}\la P\ra_{ r,D(s, r)}+\| X_P\|_{ r,D(s, r)}+\frac{\alpha}{M}\Big(\la P\ra_{ r,D(s, r)}^{\mathfrak{L}}+\| X_P\|_{ r,D(s, r)}^{\mathfrak{L}}\Big)\leq \frac{\alpha^2\eta^2e^{-7(\frac{8}{\sigma})^{t_1}}}{Mc_0},
\end{eqnarray}
where  $t_1=\frac{\tau}{\beta-\tau}$, $c_0$ is a large constant depending only on $n,\tau$ and
$\beta$.
\subsubsection{Estimates on the new error term}
We estimate the new error term $P_+$ given by
the formula
\begin{eqnarray}\label{p++}
 P_+=(P-R)\circ X_F^1+\int_0^1\{R(t),F\}\circ X_F^tdt,
\end{eqnarray}
where $R(t)=(1-t)\widehat{N}+tR$.
\begin{lemma}\label{p+}
Assume (\ref{iterationtiaojian}). Then there exists $c(n,\beta) > 0$ so that
for all $0 \leq \lambda\leq \alpha/M,$
\begin{eqnarray*}\label{4.2}
&&\la P_+\ra_{ \eta r,D(s-5\sigma, \eta r)}^{\lambda}+\| X_{P_+}\|_{ \eta r,D(s-5\sigma,\eta r)}^{\lambda}     \nonumber\\
&\preceq& \frac{c(n,\beta)e^{7(\frac{8}{\sigma})^{t_1}}}{\alpha\eta^2}\Big(\la P\ra_{ r,D(s, r)}^{\lambda}+\| X_P\|_{ r,D(s, r)}^{\lambda}\Big)^2+ \eta\Big(\la P\ra_{ r,D(s, r)}^{\lambda}+\| X_P\|_{ r,D(s, r)}^{\lambda}\Big).
\end{eqnarray*}
\end{lemma}
We divide it into two lemmas. From \cite{Pos}, we have
\begin{lemma}\label{59part1}
Assume (\ref{iterationtiaojian}), then
\begin{eqnarray*}\label{}
\| X_{P_+}\|_{ \eta r,D(s-5\sigma,\eta r)}^{\lambda}\preceq \frac{c(n,\beta)e^{7(\frac{8}{\sigma})^{t_1}}}{\alpha\eta^2}(\| X_P\|_{ r,D(s, r)}^{\lambda})^2+ \eta\| X_P\|_{ r,D(s, r)}^{\lambda}.
\end{eqnarray*}
\end{lemma}
\begin{lemma}\label{59part2}
Assume (\ref{iterationtiaojian}), then
\begin{eqnarray*}\label{}
\la P_+\ra_{ \eta r,D(s-5\sigma, \eta r)}^{\lambda}\preceq \frac{c(n,\beta)e^{7(\frac{8}{\sigma})^{t_1}}}{\alpha\eta^2}(\la P\ra_{ r,D(s, r)}^{\lambda})^2+ \eta\la P\ra_{ r,D(s, r)}^{\lambda}.
\end{eqnarray*}
\end{lemma}
\begin{proof}
By (\ref{p++}), Proposition \ref{map} and Lemma \ref{p-rsemi}, we have
\begin{eqnarray*}\label{}
\la (P-R)\circ X_F^1\ra_{ \eta r,D(s-5\sigma, \eta r)}^{\lambda}&=&\la (P-R)\circ X_F^1\ra_{ \eta r,D(s-5\sigma, \eta r)}+\lambda\la (P-R)\circ X_F^1\ra_{ \eta r,D(s-5\sigma, \eta r)}^{\mathfrak{L}}\\
&\preceq&  \la P-R\ra_{ \eta r,D(s-2\sigma, 4\eta r)}^{\lambda}\\
&\preceq&  \eta\la P\ra_{ r,D(s, r)}^{\lambda}.
\end{eqnarray*}
On the other hand, by the same method,
\begin{eqnarray}\label{rfintsemi}
\Big\la \int_0^1\{R(t),F\}\circ X_F^tdt\Big\ra_{ \eta r,D(s-5\sigma, \eta r)}^{\lambda}
&\preceq&
  \la \{R(t),F\}\ra_{ \eta r,D(s-2\sigma, 4\eta r)}^{\lambda}.
 \end{eqnarray}
 Note  $R(t)=(1-t)\widehat{N}+tR$ and $\widehat{N}=[R]$,
 from Lemma \ref{Lrfsemi} we obtain
\begin{eqnarray*}\label{}
 \la \{[R],F\}\ra_{ \eta r,D(s-2\sigma,  4\eta r)}^{\lambda}&\leq&\eta^{-2}\la \{[R],F\}\ra_{ r,D(s-2\sigma,  \frac{r}{2})}^{\lambda}\\
 &\preceq&\frac{1}{\sigma\eta^{2}}\Big(\la [R]\ra^\lambda_{ r,D(s,r)}\la F\ra^+_{ r,D(s-\sigma,r)}+\la [R]\ra_{ r,D(s,r)}\la F\ra^{+,\lambda}_{ r,D(s-\sigma,r)}\Big),
 \end{eqnarray*}
 where we use $\la \cdot \ra_{\eta r, D(s-2\sigma,\eta r)}\leq \eta^{-2}\la \cdot \ra_{r,D(s-2\sigma,r)} $ for $0<\eta<1$. Similarly
 \begin{eqnarray*}\label{}
 \la \{R,F\}\ra_{ \eta r,D(s-2\sigma,  4\eta r)}^{\lambda} &\preceq&\frac{1}{\sigma\eta^{2}}\Big(\la R\ra^\lambda_{ r,D(s,r)}\la F\ra^+_{ r,D(s-\sigma,r)}+\la R\ra_{ r,D(s,r)}\la F\ra^{+,\lambda}_{ r,D(s-\sigma,r)}\Big).
 \end{eqnarray*}
 Thus, by Lemma  \ref{fsemi} and $0\leq \lambda\leq \alpha/M$,
  \begin{eqnarray*}\label{rfintsemi2}
\nonumber(\ref{rfintsemi}) &\preceq&\frac{1}{\sigma\eta^{2}}\Big(\la P\ra^\lambda_{ r,D(s,r)}\la F\ra^+_{ r,D(s-\sigma,r)}+\la P\ra_{ r,D(s,r)}\la F\ra^{+,\lambda}_{ r,D(s-\sigma,r)}\Big)\\
\nonumber &\leq&\frac{c(n,\beta)}{\alpha\sigma\eta^{2}}\la P\ra^\lambda_{ r,D(s,r)}\la P\ra_{ r,D(s,r)}e^{2(\frac{2}{\sigma})^{t_1}}+\lambda\frac{c(n,\beta)M}{\alpha^2\sigma\eta^{2}}\la P\ra^2_{ r,D(s,r)}e^{6(\frac{8}{\sigma})^{t_1}}+\lambda\frac{c(n,\beta)}{\alpha\sigma\eta^{2}}\la P\ra_{ r,D(s,r)}\la P\ra^{\mathfrak{L}}_{ r,D(s,r)}
e^{6(\frac{8}{\sigma})^{t_1}}\\
&\leq&\frac{c(n,\beta)}{\alpha \eta^{2}}e^{7(\frac{8}{\sigma})^{t_1}}(\la P\ra^\lambda_{ r,D(s,r)})^2.
\end{eqnarray*}
\end{proof}

\subsubsection{Estimates on the frequencies.}
\begin{lemma}\label{small5}
There exists $K$ and $0<\alpha_+<\alpha$ so that
$$|\la k,\omega_+(\xi)\ra +\la l,\Omega_+(\xi)\ra|\geq \frac{\la l\ra\alpha_+}{A_k},\ \ |k|\leq K,\ |l|\leq2,$$
where $A_k=e^{|k|^{\tau/\beta}}(\beta>\tau)$.
\end{lemma}
\begin{proof}Note that $\omega_+=\omega+\widehat{\omega},\ \Omega_+=\Omega+\widehat{\Omega}.$
Since $\widehat{\omega}_j(\xi)=\frac{\partial \widehat{N}}{\partial y_j}(0,0,0,0,\xi),$
we obtain that
$$|\widehat{\omega}|_{\Pi}\leq\sup_{D(s,r)\times\Pi}\Big|\frac{\partial \widehat{N}}{\partial y}\Big|\leq\|X_{\widehat{N}}\|_{r,D(s,r)}\preceq \|X_{P}\|_{r,D(s,r)}.$$
On the other hand,
$\widehat{\Omega}_j(\xi)=\frac{\partial^2 \widehat{N}}{\partial z_j\partial \bar{z}_j}(0,0,0,0,\xi),$ thus
$$\|\widehat{\Omega}\|_{2\beta,\Pi}\leq\sup_{D(s-\sigma,r)\times\Pi}\Big|\frac{\partial^2 \widehat{N}}{\partial z_j\partial \bar{z}_j}\Big|(1+\ln j)^{2\beta}\leq\la \widehat{N}\ra_{r,D(s-\sigma,r)}\preceq\la P\ra_{r,D(s,r)}.$$
Therefore,
$$|\widehat{\omega}|_{\Pi}+\|\widehat{\Omega}\|_{2\beta,\Pi}\preceq\|X_{P}\|_{r,D(s,r)}+\la P\ra_{r,D(s,r)}.$$
\indent Similarly, for the Lipschitz norms we obtain
\begin{eqnarray}\label{omegahat}
|\widehat{\omega}|^{\mathfrak{L}}_{\Pi}+\|\widehat{\Omega}\|^{\mathfrak{L}}_{2\beta,\Pi}\preceq\|X_{P}\|^{\mathfrak{L}}_{r,D(s,r)}+\la P\ra^{\mathfrak{L}}_{r,D(s,r)}.
\end{eqnarray}
Discussing different cases we easily obtain
$$|\la k,\widehat{\omega}\ra +\la l,\widehat{\Omega}\ra|\preceq|k|(\|X_{P}\|_{r,D(s,r)}+\la P\ra_{r,D(s,r)}).$$
If we choose $\displaystyle\widehat{\alpha}\geq CK\max_{|k|\leq K}A_K\big(\|X_{P}\|_{r,D(s,r)}+\la P\ra_{r,D(s,r)}\big),$
then for $|k|\leq K,$
\begin{eqnarray*}\label{}
|\la k,\omega_+(\xi)\ra +\la l,\Omega_+(\xi)\ra|&\geq& |\la k,\omega(\xi)\ra +\la l,\Omega(\xi)\ra|-|\la k,\widehat{\omega}(\xi)\ra +
\la l,\widehat{\Omega}(\xi)\ra|\\
&\geq&\frac{\la l\ra\alpha}{A_k}- C|k|\big(\|X_{P}\|_{r,D(s,r)}+\la P\ra_{r,D(s,r)}\big)\\
&\geq&\frac{\la l\ra\alpha_+}{A_k}
\end{eqnarray*}
with $\alpha_+=\alpha-\widehat{\alpha}$.\\
\indent It remains to show that $\alpha_+>0.$ This will be done in the KAM iteration below(see (\ref{diffalpha})).
\end{proof}
\subsubsection{The iterative lemma.}
 Denote $P_0 = P$ and $N_0 = N$. Then at the $\nu-$th step of the
Newton scheme, we have a Hamiltonian $H_\nu = N_\nu + P_\nu$ where
$$N_\nu=\sum_{j=1}^{n}\omega_{\nu, j}(\xi)y_j+\sum_{j\geq1}\Omega_{\nu, j}(\xi)z_j\bar{z}_j.$$
We will show that  there exists a  symplectic coordinates transformation
$\Phi_{\nu+1}: D_{\nu+1} \times\Pi_{\nu+1}\mapsto D_\nu$
such that  $H_{\nu+1} = H_{\nu} \circ\Phi_{\nu+1}= N_{\nu+1} + P_{\nu+1}$   satisfies the same assumptions   with
${\nu+1}$ in place of $\nu,$
 where  the new normal form $N_{\nu+1}$ is associated with the new
frequencies given by $\omega_{\nu+1,j}=\omega_{\nu, j}+\widehat{\omega}_{\nu, j},\ \Omega_{\nu+1,j}=\Omega_{\nu, j}+\widehat{\Omega}_{\nu, j}$ and $P_{\nu+1}$
 is given by
\begin{eqnarray*}\label{}
 P_{\nu+1}=(P_\nu-R_\nu)\circ X_{F_\nu}^1+\int_0^1\{R_\nu(t),F_\nu\}\circ X_{F_\nu}^tdt
\end{eqnarray*}
with $R_\nu(t)=(1-t)\widehat{N}_\nu+tR_\nu$.

Let $c_1$ be twice the maximum of all constants obtained during the KAM step. Set
$r_0 = r$, $s_0 = s$, $\alpha_0 = \alpha$ and $M_0 = M$. For $\nu\geq0$ set
$$\alpha_\nu=\frac{\alpha_0}{2}(1+2^{-\nu}),\ M_\nu=M_0(2-2^{-\nu}),\ \lambda_\nu=\frac{\alpha_\nu}{M_\nu},$$
and
$$\varepsilon_{\nu+1}=\frac{c_1\varepsilon_\nu^{\frac{133}{100}}}{\alpha_\nu^{\frac13}},\ \sigma_\nu=\frac{8\cdot700^{\iota-1}}{|\ln\varepsilon_\nu|^{\iota-1}},\ \eta_\nu^3=\frac{\varepsilon_\nu^{\frac{99}{100}}}{\alpha_\nu},\ s_{\nu+1}=s_\nu-5\sigma_\nu,\ r_{\nu+1}=\eta_\nu r_\nu, \beta=\iota \tau(\iota\geq 2),$$
and $D_\nu=D(s_\nu,r_\nu).$

The initial conditions are chosen in the following way: $\sigma_0 = s_0/48 \leq 1$ so that
$s_0 > s_1 > \cdots\geq s_0/2,$
and assume $\varepsilon_0\leq\gamma_0\alpha_0^5$ with $\gamma_0\leq\min\{(\frac{1}{8Mc_0})^4,c_2(s_0),(\frac{1}{4c_1})^{10}\}$ where
 $c_2(s_0)=\exp\{-\frac{48\cdot8\cdot700^{\iota-1}}{s_0}\}.$
Furthermore, we define
 $K_\nu=K_0(\frac{36}{25})^{\nu}$\ with\ $K_0=\frac{1}{4}\ln^2(\frac{1}{4c_1\gamma_0}).$
\begin{lemma}\label{Iterative lemma}
 (Iterative lemma). Suppose that $H_\nu = N_\nu + P_\nu$ is given on $D_\nu \times\Pi_\nu$,
where $N_\nu  = \sum\limits_{
1\leq j\leq n}
\omega_{_\nu ,j} (\xi)y_j +\sum\limits_{j\geq 1}
\Omega_{_\nu ,j}(\xi)z_j \bar{z}_j$ is a normal form satisfying
\begin{equation}\label{5.1}
|\omega_\nu|_{\Pi_\nu}^{\mathfrak{L}}+\|\Omega_\nu\|_{2\beta,\Pi_\nu}^{\mathfrak{L}}\leq M_\nu,
\end{equation}
 $$|\la k,\omega_\nu(\xi)\ra +\la l,\Omega_\nu(\xi)\ra|\geq \frac{\la l\ra\alpha_\nu}{A_k},\ \ (k,l)\in\mathcal{Z},$$
on $\Pi_\nu$ and
\begin{eqnarray}\label{5.2}
\la P\ra_{ r_\nu,D_\nu}^{\lambda_\nu}+\| X_{P}\|_{ r_\nu,D_\nu}^{\lambda_\nu}\leq \varepsilon_\nu.
\end{eqnarray}
Then there exists a Lipschitz family of real analytic symplectic coordinates transformations
$\Phi_{\nu+1}: D_{\nu+1} \times\Pi_{\nu+1}\mapsto D_\nu$ with a closed subset $\Pi_{\nu+1}=\Pi_\nu\setminus\bigcup_{|k|>K_\nu}\mathcal{R}_{kl}^{{\nu+1}}(\alpha_{\nu+1})$ of $\Pi_\nu,$ where
$$\mathcal{R}_{kl}^{{\nu+1}}(\alpha_{\nu+1})=\Big\{\xi\in\Pi_\nu:\ |\la k,\omega_{\nu+1}(\xi)\ra +\la l,\Omega_{\nu+1}(\xi)\ra|< \frac{\la l\ra\alpha_{\nu+1}}{A_k}\Big\}$$
such that for $H_{\nu+1} = H_{\nu} \circ\Phi_{\nu+1}= N_{\nu+1} + P_{\nu+1}$, the same assumptions (\ref{5.1}) and (\ref{5.2}) are satisfied with
${\nu+1}$ in place of $\nu.$
\end{lemma}
\begin{proof}
By induction one verifies that
 $$\varepsilon_\nu\leq\frac{\alpha_\nu^2\eta_\nu^2e^{-7(\frac{8}{\sigma_\nu})^{t_1}}}{M_\nu c_0},\ t_1=\frac{\tau}{\beta-\tau}=\frac{1}{\iota-1}.$$
  So the smallness condition (\ref{iterationtiaojian}) at the $\nu$-th
 KAM step is satisfied, and there exists a transformation $\Phi_{\nu+1}: D_{\nu+1} \times\Pi_{\nu+1}\mapsto D_\nu$
taking $H_\nu$ into $H_{\nu+1} = N_{\nu+1} + P_{\nu+1}$. From Lemma \ref{p+}, the new error term satisfies the estimate
\begin{eqnarray*}\label{}
\la P_{\nu+1}\ra_{ r_{\nu+1},D_{\nu+1}}^{\lambda_{\nu+1}}+\| X_{P_{\nu+1}}\|_{ r_{\nu+1},D_{\nu+1}}^{\lambda_{\nu+1}}\leq \frac{c_1}{2}\Big(\frac{e^{7\cdot(\frac{8}{\sigma_\nu})^{t_1}}}{\alpha_\nu\eta_\nu^2}\varepsilon_\nu^2+\eta_\nu\varepsilon_\nu\Big)\leq \varepsilon_{\nu+1}.
\end{eqnarray*}
In view of (\ref{5.1}) the Lipschitz semi-norm of the new frequencies is bounded by
\begin{eqnarray*}\label{}
|\omega_{\nu+1}|_{\Pi_\nu}^{\mathfrak{L}}+\|\Omega_{\nu+1}\|_{2\beta,\Pi_\nu}^{\mathfrak{L}}&\leq& M_\nu+|\widehat{\omega}_\nu|_{\Pi_\nu}^{\mathfrak{L}}+\|\widehat{\Omega}_\nu\|_{2\beta,\Pi_\nu}^{\mathfrak{L}}\\
&\leq& M_\nu+\frac{c_1}{2}\varepsilon_\nu\frac{M_\nu}{\alpha_\nu}\leq M_\nu(1+\frac{1}{2^{\nu+2}})\\
&\leq& M_{\nu+1},
\end{eqnarray*}
where the second inequality is from (\ref{omegahat}). \\
\indent Finally, one verifies that $
\alpha_\nu-\alpha_{\nu+1}\geq c_1K_\nu e^{K_\nu^{\tau/\beta}}\varepsilon_\nu,$ hence
\begin{eqnarray}\label{diffalpha}
\alpha_\nu-\alpha_{\nu+1}\geq c_1K_\nu e^{K_\nu^{\tau/\beta}}\big(\la P\ra_{ r_\nu,D_\nu}^{\lambda_\nu}+\| X_{P}\|_{ r_\nu,D_\nu}^{\lambda_\nu}\big).
\end{eqnarray}
Therefore, by Lemma \ref{small5}, the small divisor estimates hold for the new frequencies with parameter
$\alpha_{\nu+1}$ up to $|k| \leq K_\nu$. Removing from $\Pi_\nu$ the union of the resonance zones
$\mathcal{R}_{kl}^{{\nu+1}}(\alpha_{\nu+1})$ for $|k| > K_\nu$ we obtain the parameter domain $\Pi_{\nu+1}\subset\Pi_\nu$  with the
required properties.
\end{proof}

\subsubsection{Proof of Theorem \ref{KAMtheorem}.} We follow the proofs in   \cite{GT} and \cite{Pos}. For readers' convenience, we use the same notations as in \cite{GT}. Firstly as \cite{Pos}, we have the estimates.
\begin{lemma}\label{convergelemma}
For $\nu\geq0,$
\begin{eqnarray*}
\frac{1}{\sigma_\nu}\|\Phi_{\nu+1}-id\|_{r_\nu,D_{\nu+1}}^{\lambda_\nu},\quad\|D\Phi_{\nu+1}-I\|_{r_\nu,r_\nu,D_{\nu+1}}^{\lambda_\nu}&\leq&{c_1e^{4(\frac{4}{\sigma_\nu})^{t_1}}}\alpha_\nu^{-1}\varepsilon_\nu,\\
|\omega_{\nu+1}-\omega_\nu|_{\Pi_\nu}^{\lambda_\nu}, \quad\|\Omega_{\nu+1}-\Omega_\nu\|_{2\beta,\Pi_\nu}^{\lambda_\nu}&\leq& c_1\varepsilon_\nu.
\end{eqnarray*}
\end{lemma}
Now suppose the assumptions of Theorem \ref{KAMtheorem} are satisfied. To apply the iterative lemma(Lemma \ref{Iterative lemma}) with $\nu=0$, set $s_0=s,\ r_0=r,\ \ldots,\ N_0=N,\ P_0=P$ and $\gamma_0=\gamma,\ \alpha_0=\alpha,\ M_0=M.$ The smallness condition is satisfied, because
\begin{eqnarray}\label{smallness}
\varepsilon=\la P\ra_{ r_0,D_0}^{\lambda_0}+\| X_{P}\|_{ r_0,D_0}^{\lambda_0}\leq \gamma_0\alpha_0^5=\varepsilon_0.
\end{eqnarray}
\indent The small divisor conditions are satisfied by setting $\Pi_0=\Pi\setminus\bigcup_{k,l}\mathcal{R}_{kl}^{{0}}(\alpha_{0}).$
Then the iterative lemma applies, and we obtain a decreasing sequence of domains $D_\nu\times\Pi_\nu$ and transformations $\Phi^{\nu}=\Phi_1\circ\cdots\circ\Phi_\nu:\ D_\nu\times\Pi_{\nu-1}\rightarrow D_{\nu-1}$ for $\nu\geq1,$ such that $H\circ\Phi^\nu=N_\nu+P_\nu.$ Moreover the estimates in Lemma \ref{convergelemma} hold.\\
\indent
From Lemma \ref{convergelemma} we have
$
\|D\Phi_{\nu}\|_{r_{\nu-1},r_{\nu-1},D_{\nu}}^{\lambda_\nu}\leq 1+{c_1e^{4(\frac{4}{\sigma_{\nu-1}})^{t_1}}}\alpha_{\nu-1}^{-1}\varepsilon_{\nu-1},
$ and thus
\begin{eqnarray}\label{dphinu}
\|D\Phi^{\nu}\|_{r_{0},r_{\nu},D_{\nu}}\leq\Pi_{n=0}^\infty(1+2^{-n-2})\leq2,
\end{eqnarray}
for all $\nu\geq0.$ Similarly we have
$
\|D\Phi^{\nu}\|^\mathfrak{L}_{r_{0},r_{\nu},D_{\nu}}\leq2.
$
Thus,
$
\|\Phi^{\nu+1}-\Phi^{\nu}\|^{\lambda_0}_{r_{0},D_{\nu+1}}\preceq\|\Phi_{\nu+1}-id\|^{\lambda_\nu}_{r_{\nu},D_{\nu+1}}.
$
So $\Phi^{\nu}$ converge uniformly on $\bigcap D_\nu\times\Pi_\nu=D(s/2)\times\Pi_\alpha$ to a Lipschitz continuous family of real analytic torus embeddings $\Phi:\ \T^n\times\Pi_\alpha\rightarrow\mathcal{P}^p$.\\
\indent From (\ref{dphinu}) we obtain
$
\|\Phi^{\nu+1}-\Phi^{\nu}\|_{r_{0},D_{\nu+1}}\leq2c_1\alpha_\nu^{-1}\varepsilon_\nu^{\frac{99}{100}}.
$
It follows
\begin{eqnarray*}\label{}
\|\Phi^{\nu+1}-id\|_{r_{0},D(s/2)\times\Pi_\alpha}\leq \sum_{n=0}^\nu 2c_1\alpha_\nu^{-1}\varepsilon_\nu^{\frac{99}{100}}\preceq \alpha_0^{-1}\varepsilon_0^{\frac{99}{100}}\preceq\varepsilon_0^{\frac{1}{2}}.
\end{eqnarray*}
Notice (\ref{smallness}), the estimate (\ref{thmestimate1}) holds on $D(s/2)\times\Pi_\alpha$. The similar discussion in \cite{GT} shows us that the estimate  (\ref{thmestimate1}) can be extended to the domain $D(s/2,r/2).$ The estimates (\ref{thmestimate2}) and (\ref{thmestimate3}) are simple and we omit the details.\\
\indent Note that  $\Phi$ is analytic on $D(s/2,r/2)$, we deduce  that $H\circ\Phi=N^*+P^*$ is analytic on  $D(s/2,r/2)$. We need to prove that
$
\partial_yP^*=\partial_zP^*=\partial_{\bar{z}}P^*=0,\quad\partial^2_{z_iz_j}P^*=\partial^2_{z_i\bar{z}_j}P^*=\partial^2_{\bar{z}_i\bar{z}_j}P^*=0
$
on $D(s/2)\times\Pi_\alpha$.  In the following we only give the proof for
$
 \partial^2_{z_i\bar{z}_j}P^* =0
$
and omit the proofs for the others.\\
\indent Note that $\|\partial^2_{z_i\bar{z}_j}P_\nu\|_{D(s/2)}\leq\varepsilon_\nu$ and
 $\|\partial^2_{z_i\bar{z}_j}(P_\nu-P_{\nu+1})\|_{D(s/2)}\preceq\varepsilon_\nu+\varepsilon_{\nu+1}\preceq\varepsilon_\nu$. It then follows
\begin{eqnarray*}\label{}
\|\partial^2_{z_i\bar{z}_j}(P_\nu-P^*)\|_{D(s/2)} \leq\sum_{k=\nu}^\infty\|\partial^2_{z_i\bar{z}_j}(P_k-P_{k+1})\|_{D(s/2)}\preceq\varepsilon_\nu
\end{eqnarray*}
and so,
{$$\|\partial^2_{z_i\bar{z}_j}P^*\|_{D(s/2)}\leq \|\partial^2_{z_i\bar{z}_j}P_\nu\|_{D(s/2)}+\|\partial^2_{z_i\bar{z}_j}(P_\nu-P^*)\|_{D(s/2)} \preceq\varepsilon_\nu$$
for all $\nu$ which means that $\partial^2_{z_i\bar{z}_j}P^*=0$ on $D(s/2)\times \Pi_{\alpha}$.}\qed
\section{Measure estimates}\label{S6}
In this section we  prove the measure estimates.
\begin{Theorem}\label{Meastheorem}
Let $\omega_\nu,\ \Omega_\nu$ for $\nu\geq0$ be Lipschitz maps on $\Pi$ satisfying
$$|\omega_\nu-\omega|,\ \|\Omega_\nu-\Omega\|_{2\beta}\leq\alpha,\ \ |\omega_\nu-\omega|^{\mathfrak{L}},\ \|\Omega_\nu-\Omega\|^{\mathfrak{L}}_{2\beta}\leq \frac{1}{2L},$$
and define the sets $\mathcal{R}_{kl}^\nu(\alpha)$ as in Lemma \ref{Iterative lemma} choosing $\tau\geq n+2.$
Then
$$Meas(\Pi\setminus\Pi_\alpha)\leq Meas(\bigcup\mathcal{R}_{kl}^\nu(\alpha))\rightarrow0, \qquad  as\ \alpha\rightarrow0.$$
\end{Theorem}
 In estimating the measure of the resonance zones it is not necessary to distinguish between the various perturbations $\omega_\nu$ and $\Omega_\nu$ of the frequencies, since only the size of the perturbation matters. Therefore, we write $\omega',\ \Omega'$ for all of them, and we have
$$|\omega'-\omega|,\ \|\Omega'-\Omega\|_{2\beta}\leq\alpha,\ \ |\omega'-\omega|^{\mathfrak{L}},\ \|\Omega'-\Omega\|^{\mathfrak{L}}_{2\beta}\leq \frac{1}{2L}.$$
Similarly, we write $\mathcal{R}_{kl}'$ rather than $\mathcal{R}_{kl}^\nu$ for the various resonance zones.
The proof of Theorem \ref{Meastheorem} requires a couple of lemmas.
\begin{lemma}\label{514}  For $l\in\Lambda=\{l:\ 1\leq |l|\leq2\}$,
$$\ln (1+\la l\ra)\geq \frac{1}{8}\|l\|_{2\beta}^{\frac{1}{2\beta}}\|l\|_{-2\beta}^{\frac{1}{2\beta}},$$
where $\|l\|_{\pm2\beta}=\sup_{j\geq1}|l_j|(1+\ln j)^{\pm2\beta}.$
\end{lemma}
\begin{proof}
We only prove the most complicated case, i.e., $\displaystyle l=(\cdots,1,\cdots,-1,\cdots)$.
In other words, $l_i=1,\ l_j=-1$ with $i<j$. Set $b=\la l\ra=j-i$. \\
Case 1: $b\geq 2e$. Clearly, $$\|l\|_{2\beta}\|l\|_{-2\beta}=\Big(\frac{1+\ln j}{1+\ln i}\Big)^{2\beta}\leq (1+\ln (i+b))^{2\beta}.$$
If $b\geq i,$ then $\|l\|_{2\beta}\|l\|_{-2\beta}\leq 2^{2\beta}(\ln b)^{2\beta}.$
It follows $\ln \la l\ra\geq\frac12 \|l\|_{2\beta}^{\frac{1}{2\beta}}\|l\|_{-2\beta}^{\frac{1}{2\beta}}.$
If $ b\leq i$, it follows $i\leq j\leq2i.$ From a straightforward computation,
$$\|l\|_{2\beta}\|l\|_{-2\beta}\leq\Big(\frac{1+\ln 2i}{1+\ln i}\Big)^{2\beta}\leq 2^{2\beta}.$$
We obtain
$\ln \la l\ra\geq1\geq\frac12 \|l\|_{2\beta}^{\frac{1}{2\beta}}\|l\|_{-2\beta}^{\frac{1}{2\beta}}$.\\
Case 2: $1\leq b\leq 2e$. Similarly,
$\frac14\|l\|_{2\beta}^{\frac{1}{2\beta}}\|l\|_{-2\beta}^{\frac{1}{2\beta}}\leq 1$.
It follows
$$\ln (1+\la l\ra)\geq \ln 2\geq\frac12 \geq \frac18\|l\|_{2\beta}^{\frac{1}{2\beta}}\|l\|_{-2\beta}^{\frac{1}{2\beta}}.$$
\indent For other cases the proofs are similar.
\end{proof}

\begin{lemma}
If $\mathcal{R}_{kl}'\neq\phi$ and $k\neq0,\ l\in\Lambda$, then $|k|\geq c_3\la l\ra,$ where $0<\alpha\leq\min\{1,\frac{1}{2}|a_1|\},$ $c_3$ is a constant depending on $a_1,M,M_1$, where $a_1$ are defined in { Assumption $\mathcal{A}$}.\end{lemma}
\begin{proof}
Case 1: $\displaystyle l=(\cdots,1,\cdots,-1,\cdots).$ In other words, $l_i=1,\ l_j=-1,\ i<j.$
\begin{eqnarray*}
\la k,\omega'\ra+\la l,\Omega'\ra&=&\underbrace{(\overline{\Omega}_i-\overline{\Omega}_j)}_{I_1}+\underbrace{( {\Omega}_i-\overline{\Omega}_i)-( {\Omega}_j-\overline{\Omega}_j)}_{I_2}
+\underbrace{\la l,\Omega'-\Omega\ra}_{I_3}+\underbrace{\la k,\omega'\ra}_{I_4},
\end{eqnarray*}
where
$|I_1|\geq|a_1||i-j|=|a_1|\la l\ra$ and
$|I_2|\leq 2 M_1$(note $\delta<0$).
From $\|\Omega'-\Omega\|_{2\beta}\leq\alpha$, it follows
$|I_3|= |\la l,\Omega'-\Omega\ra|\leq \alpha(1+\ln i)^{-2\beta}+\alpha(1+\ln j)^{-2\beta}\leq2\alpha$.
On the other hand,  $|I_4|\leq|k|(\alpha+M)$. Thus,
$|I_2+I_3+I_4|\leq2M_1+2\alpha+|k|(\alpha+M)$.
If $\mathcal{R}_{kl}'\neq\phi$, then there exists $\xi\in\Pi$ so that
 $$|\la k,\omega'(\xi)\ra +\la l,\Omega'(\xi)\ra|< \frac{\alpha\la l\ra}{A_k}.$$
Thus,
\begin{eqnarray*}
\frac{\alpha \la l\ra}{A_k}&>&|\la k,\omega'\ra+\la l,\Omega'\ra|\\
&\geq&|I_1|-|I_2+I_3+I_4|\\
&\geq&|a_1|\la l\ra-(2M_1+2\alpha+|k|(\alpha+M)).
\end{eqnarray*}
If $0<\alpha\leq \min\{1,\frac{1}{2}|a_1|\}$, then
$\la l\ra\leq\frac{2|k|(2M_1+M+3)}{|a_1|}$
or $|k|\geq\frac{|a_1|}{2(2M_1+M+3)}\la l\ra : =  \la l\ra c_3$.
For other cases the proofs are similar.
\end{proof}
From \cite{Pos}, we have
\begin{lemma}
If $|k|\geq8LM\|l\|_{-2\beta}$, then
 {$$Meas(\mathcal{R}'_{kl}(\alpha))\leq \frac{\alpha c_4}{A_k},$$
where $c_4=C_nL^nM^{n-1}\rho^{n-1}c_3^{-1}$ with $\rho=diam (\Pi).$}
\end{lemma}
Similar as \cite{Pos}, let
\begin{eqnarray}\label{lkstar}
L_*=\frac{LM}{c_3C_\beta},\ \ K_*=8LM\max_{\|l\|_{2\beta}\leq L_*}\|l\|_{-2\beta}+c(\tau,\iota,\delta)
\end{eqnarray}
with $c(\tau,\iota,\delta)$ defined in (\ref{ctautdef}) below,
we have
\begin{lemma}\label{rrrr}
If $|k|\geq K_*$ or $\|l\|_{2\beta}\geq L_*,\ l\in\Lambda,$ then for $k\neq0,$
 {
$$Meas(\mathcal{R}'_{kl}(\alpha))\leq \frac{\alpha c_4}{A_k}.$$}
\end{lemma}
\begin{proof}
The proof is followed by Lemma \ref{514} and a straightforward computation.
\end{proof}
\begin{remark}\label{0melnicov}
The same holds for $k\neq0,l=0.$
\end{remark}
Next we consider the ``resonance classes" $\mathcal{R}'_{k}(\alpha)=\bigcup_{l\in\Lambda}^*\mathcal{R}'_{kl}(\alpha)$, where the star indicates that we exclude the finitely many resonance zones with $0\leq|k|<K_*$ and $0<\|l\|_{2\beta}<L_*.$ Without loss of generality we suppose $-1\leq \delta<0$. If $\delta<-1$, then we set $\delta=-1$.
 {\begin{lemma}\label{518}
$\d Meas(\mathcal{R}'_{k}(\alpha))<\frac{c_5\alpha^\mu}{|k|^{\tau-1}}$, where $\mu=\frac{\delta}{\delta-1}.$
\end{lemma}}
\begin{proof}
Write $\Lambda=\Lambda^+\bigcup\Lambda^-$, where $\Lambda^-$ contains those $l\in\Lambda$   with two nonzero components of the opposite sign, and  $\Lambda^+$ contains the rest. It is easy to obtain
$Meas(\bigcup_{l\in\Lambda^+}^*\mathcal{R}'_{kl}(\alpha))\leq \frac{c_6|k|^2\alpha}{A_k}$.\\
\indent Now we turn to the minus case.  For $l\in\Lambda^-$, we have $\la l,\Omega'\ra=\Omega'_i-\Omega'_j$ and $\la l\ra=|i-j|,$ and up to an irrelevant sign, $l$ is uniquely determined by  two integers $i\neq j.$ We may suppose that $i-j=b>0.$ Then, for $|k|\geq K_*\geq([(\tau+1)\iota]+1)!$,
\begin{eqnarray*}
|\la k,\omega'\ra+a_1b|&<&\frac{\alpha b}{A_k}+2\alpha(1+\ln j)^{-2\beta}+2M_1j^\delta\\
&\leq &\frac{\alpha b}{|k|^\tau}+2\alpha(1+\ln j)^{-2\beta}+2M_1j^\delta.
\end{eqnarray*}
Therefore,
\begin{eqnarray*}
\mathcal{R}'_{kij}(\alpha)&=&\Big\{\xi:\ |\la k,\omega'(\xi)\ra +\la l,\Omega'(\xi)\ra|< \frac{\alpha b}{A_k}\Big\}\\
&\subset&Q_{kbj}:= \Big\{\xi:\ |\la k,\omega'(\xi)\ra +a_1b\ra|< \frac{\alpha b}{|k|^\tau}+2\alpha(1+\ln j)^{-2\beta}+2M_1j^\delta\Big\}.
\end{eqnarray*}
Moreover, $Q_{kbj}\subset Q_{kbj_0}$ for $j\geq j_0.$ For fixed $b\leq c_3^{-1}|k|$,  we obtain
\begin{eqnarray}\label{rrr}
\nonumber Meas(\bigcup_{i-j=b}^*\mathcal{R}'_{kij}(\alpha))
&\leq&\sum_{j<j_0}Meas(\mathcal{R}'_{kij}(\alpha))+ Meas(Q_{kbj_0})\\
&\leq&\frac{c_4\alpha j_0}{A_k}+c_6\Big(\frac{\alpha b}{|k|^{\tau+1}}+\frac{2\alpha(1+\ln j_0)^{-2\beta}}{|k|}+\frac{2M_1j_0^\delta}{|k|}\Big).
\end{eqnarray}
Choose $j_0=\max\{\exp(|k|^{\frac{\tau-1}{2\beta}}),\ \alpha^\frac{\gamma}{\delta}|k|^{\frac{1-\tau}{\delta}}\}$, where $\gamma$ will be fixed in the end. By computation, if choose
\begin{eqnarray}\label{ctautdef}
|k|\geq K_*\geq c(\tau,\iota,\delta) : =\max\Big\{2^{2\iota},\big([\iota(\tau+1+\frac{1-\tau}{\delta})]+1\big)!\Big\},
\end{eqnarray}
then
$$(\ref{rrr})\leq c_7\Big(\frac{\alpha^{1+\frac{\gamma}{\delta}}}{|k|^\tau}+\frac{\alpha }{|k|^\tau}+\frac{\alpha^{ \gamma}}{|k|^\tau}\Big).$$
Note $0<-\delta\leq1$, if choose $\gamma=\frac{\delta}{\delta-1},$ then
$(\ref{rrr})\leq c_8\frac{\alpha^{\mu} }{|k|^\tau}$. Summing over $b$,
$Meas(\bigcup_{l\in\Lambda^-}^*\mathcal{R}'_{kl}(\alpha))\leq c_9\frac{\alpha^{\mu} }{|k|^{\tau-1}}$.
Thus
$$Meas(\bigcup_{l\in\Lambda}^*\mathcal{R}'_{kl}(\alpha))\leq \frac{c_6|k|^2\alpha}{A_k}+c_9\frac{\alpha^{\mu} }{|k|^{\tau-1}}\leq c_{5}\frac{\alpha^{\mu} }{|k|^{\tau-1}}.$$
\end{proof}
From Remark  \ref{0melnicov}, if $k\neq0,\ l_0=0$, $Meas(\mathcal{R}'_{kl_0}(\alpha))\leq \frac{c_4\alpha}{A_k}$, where
we define $\mathcal{R}''_{k}(\alpha)=\bigcup\limits_{|k|\geq K_*}\mathcal{R}'_{kl_0}(\alpha)$. Note the choice of $K_*$ and $|k|\geq K_*$, we deduce that for $l_0=0$,
$Meas(\mathcal{R}'_{kl_0}(\alpha))\leq \frac{c_4\alpha}{|k|^{\tau-1}}\leq  \frac{c_4\alpha^\mu}{|k|^{\tau-1}}$. Thus we have
\begin{lemma}For $|k|\geq K_*$,
$Meas(\mathcal{R}'_{k}(\alpha)\cup \mathcal{R}''_{k}(\alpha))\leq \frac{c_{10}\alpha^\mu}{|k|^{\tau-1}}$.
\end{lemma}
\begin{lemma}\label{520}
There exists a finite subset $\mathcal{X}_1\subset\mathcal{Z}$ and a constant $\tilde{c}_1$ such that
$$Meas(\bigcup_{(k,l)\notin\mathcal{X}_1}\mathcal{R}_{kl}^\nu(\alpha))\leq\tilde{c}_1\rho^{n-1}\alpha^\mu$$
for all sufficiently small $\alpha$.  The constant $\tilde{c}_1$ and the index set $\mathcal{X}_1$ are monotone functions of the domain $\Pi$: they do not increase for closed subsets of $\Pi$. In particular,
$$\mathcal{X}_1\subset\{(k,l):\ 0\leq|k|<\widetilde{K}_{*} : = 16LM+c(\tau,\iota,\delta),\ 0< \|l\|_{2\beta}\leq L_*, l\in \Lambda\}.$$
\end{lemma}
\begin{proof}
The proof is from Lemma \ref{518} with $\tau\geq n+2$ and similar as \cite{Pos}.
\end{proof}
If as above we set $l_0=0$, then we obtain a similar lemma as above.
\begin{lemma}
There exists a finite subset $\mathcal{X}_2\subset\mathcal{Z}$ and a constant $\tilde{c}_2$ such that
$$Meas(\bigcup_{(k,l_0)\notin\mathcal{X}_2}\mathcal{R}_{kl_0}^\nu(\alpha))\leq\tilde{c}_2\rho^{n-1}\alpha^\mu$$
for all sufficiently small $\alpha$.
The constant $\tilde{c}_2$ and the index set $\mathcal{X}_2$ are monotone functions of the domain $\Pi$: they do not increase for closed subsets of $\Pi$. In particular,
$\mathcal{X}_2\subset\{(k,l_0):\ 0\leq|k|<\widetilde{K}_{*}\}$.
\end{lemma}
\noindent {\bf Proof of Theorem \ref{Meastheorem}}:
If we choose $$\gamma_0\leq\min\Big\{\Big(\frac{1}{4c_1}\Big)^{10},\ \Big(\frac{1}{8c_0M}\Big)^{4},\ c_2(s_0),\ \frac{1}{4c_1}\exp(-2\widetilde{K_*}^{\frac12})\Big\},$$
then $K_0\geq\widetilde{K_*}.$ Thus, when $\nu\geq 1$ and $l\in \Lambda$,
$Meas(\bigcup\limits_{(k,l)\in \mathcal{X}_1}\mathcal{R}_{kl}^{\nu}(\alpha))= 0$. Since $\mathcal{X}_1$ is finite,  { Assumption $\mathcal{A}$}  implies
$Meas(\bigcup\limits_{(k,l)\in \mathcal{X}_1}\mathcal{R}_{kl}(\alpha))\rightarrow  0 $ as $\alpha\rightarrow 0$. Combined with Lemma \ref{520}, we have
\begin{eqnarray}\label{lambda1}
Meas\Big(\bigcup\limits_{l\in \Lambda}\mathcal{R}_{kl}^\nu(\alpha)\Big)\rightarrow 0, \qquad {\rm as }\ \alpha\rightarrow 0.
\end{eqnarray}
The proof for $l_0=0$ is similar. We have
\begin{eqnarray}\label{lambda2}
Meas\Big(\bigcup\limits_{k\neq 0}\mathcal{R}_{kl_0}^\nu(\alpha)\Big)\rightarrow 0, \qquad {\rm as }\ \alpha\rightarrow 0.
\end{eqnarray}
Combined with (\ref{lambda1}) and (\ref{lambda2}) we complete the proof.\qed

\section{Appendix}
\begin{lemma}\label{basic1}
For $j\geq 1$ and $\beta> 1$, there exists a constant $C(\beta)$ independent of $j$ such that,
$$
\sum\limits_{l\geq 1}\frac{1}{(1+|j-l|)(1+\ln l)^{\beta}}\leq C(\beta).$$
\end{lemma}
\begin{proof} The summation is divided into three parts,
$$(\sum\limits_{1\leq l\leq j/2}+\sum\limits_{j/2<l\leq j}+\sum\limits_{l> j})\frac{1}{(1+|j-l|)(1+\ln l)^{\beta}}:=(I)+(II)+(III).$$
Hence the result is followed by the following facts:
\begin{eqnarray*}
(I)&\leq & \frac{j}{2}\frac{1}{\frac{j}{2}(1+\ln 2)^{\beta}}\leq C(\beta),\\
(II)&\leq&  \sum\limits_{1\leq k\leq j/2}\frac{1}{k(1+\ln (j-k-1))^{\beta}}\leq  \sum\limits_{1\leq k\leq j/2}\frac{1}{k(1+\ln k)^{\beta}}\leq C(\beta),\\
(III)&\leq & \sum\limits_{k> 1}\frac{1}{(1+k)(1+\ln (j+k))^{\beta}}
\leq  \sum\limits_{k> 1}\frac{1}{(1+k)(1+\ln k)^{\beta}}\leq C(\beta).
\end{eqnarray*}

\end{proof}
\begin{remark}
If $\beta\geq2$, then $
\sum\limits_{l\geq 1}\frac{1}{(1+|j-l|)(1+\ln l)^{\beta}}\leq C,$
where $C$ is independent of $j$ and $\beta$.
\end{remark}

\begin{lemma}\label{basic2}
 For any $j,l\geq 1$,  $p\geq 2$ and $\beta\geq 1$,
 $$\sum_{l\geq1}\frac{(1+j)^2}{l^p(1+|j-l|)^2(1+\ln l)^{2\beta}}\leq C.$$
\end{lemma}
\begin{proof}
\begin{eqnarray*}
&&\sum_{l\geq1}\frac{(1+j)^2}{l^p(1+|j-l|)^2(1+\ln l)^{2\beta}}\\
&\leq & \sum\limits_{2l\leq j}\frac{(1+j)^2}{l^p(1+|j-l|)^2(1+\ln l)^{2\beta}}+\sum\limits_{2l>j}\frac{(1+j)^2}{l^p(1+|j-l|)^2(1+\ln l)^{2\beta}}= (I) +(II).
\end{eqnarray*}
We estimate $(I)$ and $(II)$ respectively.
For (I), note $|j-l|\geq j/2$ and $p\geq 2$, we have
$$(I)\leq \sum\limits_{2l\leq j}\frac{(1+j)^2}{l^p(1+j/2)^2}\leq 4\sum\limits_{l\geq 1}\frac{1}{l^p}\leq C.$$
For (II), from $p\geq 2$ and $\beta\geq 1$,
\begin{eqnarray*}
(II)&\leq & \sum\limits_{2l>j} \frac{(1+j)^2}{l^2(1+|j-l|)^2(1+\ln l)^{2\beta}}\\
&\leq &\sum\limits_{2l>j} \frac{(1+j)^2}{(j/2)^2(1+|j-l|)^2(1+\ln l)^{2\beta}}\\
&\leq & C\sum\limits_{l\geq 1}\frac{1}{(1+|j-l|)^2(1+\ln l)^{2\beta}}\\
\underline{{\rm Lemma}\  \ref{basic1}} &\leq & C.
\end{eqnarray*}
\end{proof}

\end{document}